\documentclass[12pt,reqno]{amsart}

\usepackage{soul}
\usepackage{multicol,amsmath,graphics, mathtools, cancel,soul,color,bbm,amsfonts,dsfont,tikz,mathrsfs,amssymb,bm,bbold}
\usepackage[left=1in, right=1in, top=1.1in,bottom=1.1in]{geometry}

\usepackage[colorlinks=true,linkcolor=blue,citecolor=blue]{hyperref}

\usepackage{libertine} \usepackage{libertinust1math}

\usepackage{enumerate}

\usepackage{comment}

\usetikzlibrary{matrix,patterns,positioning}
\numberwithin{equation}{section}

\newcommand{\Poiss}{\eta}
\newcommand{\freeWass}[1]{\mathcal{W}_{nc,#1}}
\newcommand{\classWass}[1]{W_{#1}}

\newcommand{\C}{\mathbb{C}}

\newcommand{\E}{\mathbb{E}}

\newcommand{\N}{\mathbb{N}}

\newcommand{\R}{\mathbb{R}}

\newcommand{\Z}{\mathbb{Z}}

\newcommand{\Ac}{\mathcal{A}}

\newcommand{\Fc}{\mathcal{F}}

\newcommand{\Hc}{\mathcal{H}}

\newcommand{\Mca}{\mathcal{M}}

\newcommand{\Pc}{\mathcal{P}}

\newcommand{\fcomm}[1]{\noindent\textcolor{purple}{\textbf{#1}}}

\newcommand{\radio}{\mathrm{r}}

\newcommand{\abs}[1]{\left\vert#1\right\vert}

\newcommand{\norm}[1]{\left\lVert#1\right\rVert}

\newcommand{\vertiii}[1]{{\left\vert\kern-0.25ex\left\vert\kern-0.25ex\left\vert #1
		\right\vert\kern-0.25ex\right\vert\kern-0.25ex\right\vert}}

\newcommand{\Indi}[1]{\mathbbm{1}_{{#1}}}

\newcommand{\gaussian[1]}{g^{\circledast}_{{\,_{#1}}}}

\newcommand{\OrdAsProduct}[3]{  {\textstyle  \underset{{#2} \in {#3}} {\overset{\longrightarrow}{\prod}}} \, {#1}_{#2}   }

\def\R{\mathbb{R}}

\def\dh2l{\mathbf{d}_{\mathbb{H}_{2\ell}}}
\def\d2{\mathbf{d}_2}

\def\1{\mathbbm{1}}

\theoremstyle{remark}

\theoremstyle{definition}

\newcounter{dummy} \numberwithin{dummy}{section}
\newtheorem{Definition}[dummy]{Definition}

\newtheorem{Theorem}[dummy]{Theorem}

\newtheorem{Lemma}[dummy]{Lemma}
\newtheorem{Example}[dummy]{Example}
\newtheorem{Remark}[dummy]{Remark}
\newtheorem{Notation}[dummy]{Notation}

\newtheorem*{Lemma*}{Lemma}
\newtheorem*{Theorem*}{Theorem}

\title{Distributional comparison for non-commutative infinitely divisible probability measures}
\date{\today}

\makeatletter
\@namedef{subjclassname@1991}{2020 Mathematics Subject Classification}
\makeatother

\author{Arturo Jaramillo}
\author{Josue Vazquez-Becerra}

\begin{document}

	\maketitle
We determine ``cumulant-type'' upper bounds of the non-commutative Wasserstein distance for certain classes of distributions $\mu$ and $\nu$, which are infinite divisible with respect to the Boolean, classical and free convolutions. 
%
%between a general non-commutative Boolean, classical and free compound Poisson distribution $\mu$ and a benchmark compound distribution counterpart $\nu$ with Levy measure concentrated in $m$ points. 
%
The main contribution of the manuscript is an estimation of the non-commutative Wasserstein distance between $\mu$ and $\nu$, expressed in terms of %an explicit constant multiple of 
the difference between cumulants of order less than $2m+4$. %, between $\mu$ and $\nu$.
%
%Since convergence in cumulants and convergence in moments are equivalent, the last theorem enacts a drastic simplification of the so-called method of moments. 
	
%\ \\ What and why! 

%\newpage 
	
\section{Introduction}

%\subsection{Background}

Infinitely divisible distributions are central to both non-commutative and classical
probability. 
Not only they arise as limiting distributions for sums of independent random variables, %in triangular arrays, 
but also they are the building blocks for  L\'{e}vy processes  \cite{MR1605393,MR3155252,MR3185174}, i.e., stochastic processes with independent and stationary increments. 
A main question then relates to the comparison or proximity between these distributions through a metric. 
This manuscript contributes to this subject by providing ``cumulant-type'' and ``moment-type'' upper bounds for the Wasserstein distance within two classes of distributions that are infinitely divisible with respect to the Boolean, classical, and free convolution, see Theorem \ref{thm:main_theorm}. 

Let $\mathcal{P}(\R)$ denote the set of probability distributions over $\R$. 
The convolution of measures relates intrinsically to the notion of independence:  % from probability theory: 
given two probability measures $\mu,\nu \in \mathcal{P}(\R)$, their classical convolution $\mu \ast \nu$ can be defined as the probability distribution of $x+y$ with $x$ and $y$ classically independent random variables with distributions $\mu$ and $\nu$, respectively. 
It is a fundamental result from probability theory that the convolution $\mu \ast \nu$ does not depend on any particular choice of $x$ or $y$ to prescribe it. 

Now, an approach from non-commutative probability, sometimes referred to as quantum probability, considers random variables as elements of a $C^{*}$-algebra $\mathcal{A}$  and lets a positive linear functional $\varphi:\Ac \to \C$ take the role of the expected value $\E[ \, \cdot  \, ]$. 
The pair $(\Ac,\varphi)$ is referred to as a $C^{*}$-probability space and the elements $a \in \Ac$ are called (non-commutative) random variables, see \cite[Lecture 3]{MR2266879} for general background on this approach. % and fundamental properties. 
If $\mathcal{A}$ has a multiplicative identity $1_{\Ac}$, the unital condition $\varphi[1_{\Ac}]=1$ is also required. 

It was proved in \cite{MR1919720,MR2016316,MR1426844} that there exist only five notions of independence arising from natural products in non-commutative probability spaces, namely, tensor, Boolean, free, monotone, and anti-monotone independence. %, under the above framework,
Tensor independence corresponds to classical independence from classical probability theory. 
Moreover, anti-monotone independence is a mirror image of the monotone case, so most results on anti-monotone variables are implicitly derived from their monotone counterparts. 

Each of these five notions of independence yields a different convolution on the set $\mathcal{P}(\R)$ of probability measures over $\R$, however, they all can be defined in a equivalent way as in the classical case for compactly supported distributions. 
Concretely, if $x$ and $y$ are random variables with distributions of compact support $\mu$ and $\nu$, respectively, then the convolution $\mu \circledast \nu$ is the analytical distribution of $x+y$ provided that $x$ is $\circledast$-independent from $y$. 
The classical, free, Booelan, and monotone convolutions are denoted by $\ast$, $\boxplus$, $\uplus$, and $\rhd $, respectively. 
It should be mentioned that monotone convolution $\rhd$ is the only one that is not commutative, in the sense that \emph{``$x$ monotone independent from $y$''} is not equivalent to \emph{``$y$ monotone independent from $x$''}, and consequently $\mu \rhd \nu$ and $\nu \rhd \mu$ are generally not equal. 

In the framework of $C^{*}$-probability spaces $(\Ac,\varphi)$, a self-adjoint random variable $a = a^\ast \in \Ac $ has  analytical distribution $\mu \in \Pc(\R)$ if the $\varphi$-moments of $a$ match the classical moments of $\mu$, i.e., for every integer $n \geq 0$ we have that 
\begin{align}\label{eqn:phi-moments-mu}
%m_n[a] :=	
\varphi[a^n] \ =  \int_{\R} t^n \, d \mu (t) 
%=: m_n[\mu]
. 
\end{align}
There always exist analytical distributions for self-adjoint random variables in $C^{*}$-probability spaces, see \cite[Proposition 3.13]{MR2266879}.
The monotone, Boolean, and free convolutions are well-defined operations on the set of probability distributions $\mathcal{P}(\R)$; these results were first proved for compactly supported measures in \cite{muraki2000monotonic,MR1426845,MR0839105}, and then extended to measures with unbounded support in \cite{MR1254116, MR2541362, MR1165862}. 
Once again, each convolution $\mu \circledast \nu$ does not depend on any particular choice of $x$ or $y$  to prescribe it. 
\begin{Definition} Let $\circledast \in \{ \ast, \boxplus ,\uplus, \rhd \}$. 
	A probability measure $\mu \in \mathcal{P}(\R)$ is said to be \emph{$\circledast$-infinitely divisible} if for each integer $n \geq 1$ there exists a probability distribution $\mu_n \in \mathcal{P}(\R)$ such that 
	\begin{align*}
	\mu = 	\underbrace{\mu_n \circledast \mu_n \circledast \cdots \circledast \mu_n}_{n \text{ times}} .
	\end{align*} 
\end{Definition} 

Let $\mathcal{ID}(\circledast)$ denote the set of $\circledast$-infinitely divisible probability measures over $\R$. 
The following theorem due to L\'{e}vy and Khintchine \cite{khintchine1937theorie,levy1938theorie}, in the classical case,  Bercovici and Pata  \cite{MR1709310}, in the the free and Boolean cases, and Anshelevich and Williams \cite{MR3214313}, in the monotone case, states that infinite divisibility emerges naturally within limit theorems for sums of independent random variables. % in triangular arrays. 
\begin{Theorem}\label{thm:ID_as_limit_arrays}
Let $\circledast \in \{ \ast, \boxplus ,\uplus, \rhd \}$. 
	A probability measure $\mu$ is $\circledast$-infinitely divisible if and only if there exist a sequence of probability measures $\{ \mu_n \}_{n \in \N}$ and a sequence of positive integers $ r_1 < r_2 < r_3 < \cdots $ such that 
$$\underbrace{\mu_n \circledast \mu_n \circledast \cdots \circledast \mu_n}_{r_n \text{ times}} \text{\quad converges to\quad} \mu \text{\quad in distribution.}$$ %$\mu$ is the limit in distribution of the sequence
	%
	%\begin{align*} 	\underbrace{\mu_n \circledast \mu_n \circledast \cdots \circledast \mu_n}_{r_n \text{ times}}  \quad \text{converges to} \quad \mu \quad \text{in istribution.}	\end{align*} 
\end{Theorem} 

Certainly, two of the most fundamental limit theorems %in probability 
are the Central Limit Theorem and the Poisson Limit Theorem. 
The limiting distributions obtained from these two limit theorems, the $\circledast$-Gaussian and the compound $\circledast$-Poisson, concern this manuscript. 
Rather than presenting %them in classical terms, %
these two infinitely divisible measures in classical terms, %as limiting measures for sums of independent random variables, 
we will formulate them %$\circledast$-Gaussian and compound $\circledast$-Poisson distributions 
in terms of $\circledast$-cumulants, which are quantities well-suited to study convolutions in non-commutative probability. 

Let $\mathcal{P}_n(\R)$ denote the set of probability distributions over $\R$ with finite moments of order $1,\ldots,n$. 
And let $\mathcal{P}_{\infty}(\R) = \cap_{n \geq 1} \mathcal{P}_n(\R) $ be the set of probability distributions over $\R$ with finite moments of all orders. 
Recall that the $n$-th moment of a probability distribution $\mu \in \Pc(\R)$ is 
\begin{align*}
m_n[\mu] := \int_{\R} \ t^k \, \mu (dt)
\end{align*}
if such integral exists.
On the other hand, the $\circledast$-cumulants $\kappa^{\circledast}_n[\mu]$ of a probability distribution $\mu \in \mathcal{P}_n(\R)$ are determined recursively by the relation 
\begin{align*}
	m_n[\mu] \ \ = \sum_{\pi \in P_{\circledast}(n)} \kappa^{\circledast}_{\pi} [\mu]  \quad \text{with} \quad 
	\kappa^{\circledast}_{\pi} [\mu] : = \prod_{V \in \pi} \kappa^{\circledast}_{\abs{V}} [\mu]
\end{align*}
where $P_{\circledast}(n)$ is certain subset of (labeled) partitions of $[n]:=\{1,2,\ldots,n\}$, see Section \ref{subsection:cumulants}. 
Each $n$-th $\circledast$-cumulant $\kappa^{\circledast}_n[\mu]$ can be expressed as linear combination of products of the first $n$ moments $m_1[\mu],\ldots,m_n[\mu]$, and vice versa. 
%
%For instance, for , we have $$ \kappa^{\circledast}_1[\mu] = m_1[\mu]  \qquad \text{and}	\qquad \kappa^{\circledast}_2[\mu] = m_2[\mu] -  m_1[\mu] m_1[\mu]. $$
%
Consequently, the sequence of $\circledast$-cumulants of all orders $\{ \, \kappa^{\circledast}_n[\mu] \, \}_{n\geq 1}$ contains exactly the same information as the sequence of moments $\{ \, m_n[\mu] \, \}_{n\geq 1}$ for probability measures $\mu \in \mathcal{P}_{\infty}(\R)$. 

The $\circledast$-cumulants are more suitable than moments to investigate convolution of measures  in non-commutative probability due to fact that 
they linearize convolution, namely, they satisfy  
\begin{align}\label{eqn:linearization_cumulants}
	\kappa^{\circledast}_n[\mu_1 \circledast \cdots \circledast \mu_r ]
=
	\kappa^{\circledast}_n[\mu_1 ]
+ \cdots +
	\kappa^{\circledast}_n[\mu_r ] 
\end{align}
for $\mu_1, \ldots, \mu_r \in \Pc(\R)$. 
No additional conditions on the measures $\mu_k$ are required for classical, free, or Booelan convolution, but the relation \eqref{eqn:linearization_cumulants} holds for monotone convolution  only if $\mu_1, \ldots, \mu_r$ are identically distributed. 
The $\circledast$-cumulants also satisfy the homogeneity property
\begin{align}\label{eqn:homogeneity_cumulants}
	\kappa^{\circledast}_n[\mathrm{Dil}_{\lambda}(\mu)] %[\mu_{\lambda}  ]
	=
	\lambda^n \, \kappa^{\circledast}_n[\mu]
\end{align}
where $\mathrm{Dil}_{\lambda}(\mu)$ is the distribution of $\lambda x$ with $x$ any random variable with distribution $\mu$. 
Monotone, free and Boolean cumulants were first introduced and investigated in  \cite{MR2884229, MR1268597,MR1426845}, respectively. 
\begin{Definition}\label{def:gaussian_variable}
Let $\circledast \in \{ \ast, \boxplus ,\uplus, \rhd \}$. 
A probability measure $\mu \in \mathcal{P}_{\infty}(\R)$ is said to be \emph{a centered $\circledast$-Gaussian with variance $\sigma^2 \geq 0$} if it holds that %$\kappa^{\circledast}_n[\mu]= \sigma^2$, if $n=2$, and  $\kappa^{\circledast}_n[\mu]= 0$, otherwise.
\begin{align}\label{eqn:gaussian_in_cumulants}
	\kappa^{\circledast}_{\ell}[\mu] = \left\{ \begin{array}{cl}
		\sigma^2, &  \text{if } {\ell}=2,\\
		0, & \text{otherwise.}
	\end{array}\right.
\end{align}
If $\mu$ is a centered $\circledast$-Gaussian with variance $\sigma^2 \geq  0$, we denoted it by $\gaussian[\sigma^2]$. 
\end{Definition} 
The fact that \eqref{eqn:gaussian_in_cumulants} gives genuine Gaussian distributions, in the sense that it characterizes the measures arising from central limit theorems, can be heuristically justified as follows. 
Suppose $\mu_n$ and  $\mu$ are as in Theorem \ref{thm:ID_as_limit_arrays}. 
One might expect that, under possibly some additional conditions, convergence in distribution yields convergence in cumulants, so we would have 
\begin{align}\label{eqn:convergence_in_cumulants}
	\kappa^{\circledast}_{\ell}[\mu] = \lim_{n \to \infty} \, \kappa^{\circledast}_{\ell} [  \underbrace{\mu_n \circledast \cdots \circledast \mu_n}_{r_n \text{ times}}  ] 
	= \lim_{n \to \infty}  \, r_n \cdot \kappa^{\circledast}_{\ell} [  \mu_n ] 
\end{align}
where the last equality follows from the linearization property \eqref{eqn:linearization_cumulants}. 
To obtain a central limit theorem, one takes $r_n=n$ and $\mu_n$ as the distribution of $\frac{1}{\sqrt{n}} x$ where $x$ is a random variable with distribution $\nu \in \Pc_{\infty}(\R)$ having mean zero and variance $\sigma^2 \geq 0$; %, i.e., %which translates to cumulants as 
these two conditions become $\kappa^{\circledast}_1[\nu] = 0$ and $\kappa_2^{\circledast}[\nu] = \sigma^{2}$ in terms of $\circledast$-cumulants. 
But then, the homogeneity property \eqref{eqn:homogeneity_cumulants} gives 
\begin{align}\label{eqn:homogeneity_clt}
r_n \cdot \kappa^{\circledast}_{\ell} [  \mu_n ]  
%	= 	n \, \kappa^{\circledast}_m [  \mu_n   ] 
	= 
	n^{1-{\ell}/2} \, \kappa^{\circledast}_{\ell} [  \nu   ]. 
\end{align}
Putting together \eqref{eqn:convergence_in_cumulants}  and \eqref{eqn:homogeneity_clt}, we then obtain \eqref{eqn:gaussian_in_cumulants}. 
%
%\jcomment{the monotone Gaussian distribution (the arcsine law), the free Gaussian distribution (the semicircle law), the Boolean Gaussian distribution (the symmetric Bernoulli law)}
%
%
%
%
%

\begin{Definition}\label{def:compound_Poisson}
Let $\circledast \in \{ \ast, \boxplus ,\uplus, \rhd \}$. 
A probability measure $\mu \in \mathcal{P}_{\infty}(\R)$ is said to be \emph{a compound $\circledast$-Poisson distribution with rate $ \lambda \geq 0 $ and jump distribution $\Poiss \in \mathcal{P}_{\infty}(\R)$} if for every integer $ {\ell} \geq 1$ we have that
%it holds that $\kappa^{\circledast}_n[\mu] = \lambda \, m_n[\Poiss]$ 
%
\begin{align}\label{eqn:cpoisson_in_cumulants}
\kappa^{\circledast}_{\ell}[\mu] = \lambda \, m_{\ell}\;\![\Poiss] .
\end{align}
\end{Definition} 
As earlier with Gaussian distributions, one can heuristically justify that \eqref{eqn:cpoisson_in_cumulants} does give genuine Poisson distributions, in the sense it characterizes the measures arising from laws of rare events. 
In this case, one considers Equation \eqref{eqn:convergence_in_cumulants} with $ \mu_n  = ( {\textstyle ( 	1-\frac{\lambda}{n} ) \,	\delta_{0} 	+\frac{\lambda}{n} \, \Poiss } )$, see Section \ref{subsection:infinitely_divisible} for more details.  
%
%we leave the rest of the details for Section \ref{section:preliminaries}. 
%
%
%
%
%
 
%\subsection{Main result}
The new results in this manuscript concern to the instance when convergence in distribution is guaranteed by convergence of the first  moments (or, equivalently, cumulants) within the set of $\circledast$-infinitely divisible distributions. % $\mathcal{ID}(\circledast)$. 
As example and precedent of this, we have the following theorem due to Arizmendi, stating that convergence of the fourth cumulant implies convergence in distribution. % within the space $\mathcal{ID}(\circledast)$.
\begin{Theorem}[{\cite[Theorem 3.1]{MR3158548}}]\label{thm:arizmendi_fourth_cumulant}
Suppose $\{\mu_{n}\}_{n \in  \N} \subset \mathcal{ID}(\circledast)$ with $\circledast \in \{ \ast, \boxplus ,\uplus, \rhd \}$ is a sequence of infinitely divisible distributions with mean zero and variance one. 
Then, the condition $\kappa_4[\mu_{n}]\rightarrow\kappa_4[g^{\circledast}_{\, 1}]$ implies the convergence in distribution of $\mu_n$ towards $g^{\circledast}_{\, 1}$.
\end{Theorem} 
It can be easily verified that $\kappa_n[\mu_{n}]\rightarrow\kappa_n[g^{\circledast}_{\, 1}]$, for $n=1,2,3,4$, under the hypothesis in Theorem \ref{thm:arizmendi_fourth_cumulant}. 
Since convergence in cumulants and convergence in moments are equivalent, the last theorem enacts a drastic simplification of the so-called method of moments. 
This type of phenomenon was first observed by Nualart and Peccati in \cite{MR2118863}, where they proved that, in a fixed Wiener chaos, convergence of the fourth moment is equivalent to convergence in distribution to a classical Gaussian distribution. 
Further extensions and related developments of the latter are commonly known as \emph{fourth-moment theorem} (or phenomenon). 
In this direction, we mention the papers \cite{MR3813981} on Wiener chaos in Poisson spaces, \cite{MR2978133} on  $\boxplus$-Wigner chaos, and \cite{MR3050508} on Markov operators. 

Having established a limit theorem, a natural follow-up question concerns the rate or speed of convergence. 
For Theorem \ref{thm:arizmendi_fourth_cumulant}, and in the classical and free cases, this question was addressed by Arizmendi and Jaramillo in \cite{MR3208324} in terms of the Kolmogorov distance. 
\begin{Theorem}[{\cite[Theorems 1.5 and 1.6]{MR3208324}}]\label{thm:arizmendijaramillo_4th_rate}
	If $\mu\in\mathcal{ID}(\circledast)$ with $\circledast\in\{*,\boxplus\}$ has mean zero and variance one, then there exists a  constant $C_{\circledast}>0$ such that 
	\begin{align}\label{eqn:example_moment_bound}
		\left\{
		d_{\text{Kol}} \, \left(  \mu , g^{\circledast}_{\, 1} \right) \right\}^{2} \leq C_{\circledast} \, \Big| \, m_4[\mu] - m_4[ g^{\circledast}_{\, 1} ] \, \Big|, 
	\end{align}
where $d_{\text{Kol}} \, (\cdot, \cdot)$ denotes the Kolmogorov distance. 
\end{Theorem}

We refer to bounds that are expressed in terms of moments (resp., cumulants) as \emph{moment-type} (resp., \emph{cumulant-type}). % in this manuscript.  
The inequality \eqref{eqn:example_moment_bound} is then an example of a moment-type upper bound. 
In this manuscript, we succeed in providing cumulant-type upper bounds that extend the fourth-moment phenomenon to the non-Gaussian regime, at the cost of increasing the number of moments-cumulants considered. 
Another precedent to our work is the following theorem due to Arizmendi.

\begin{Theorem}[{\cite[Theorem 4.1]{MR3158548}}]\label{thm:arizmendi_cumulants_convergence}
Let $\mu_{\infty}\in \mathcal{ID}(\circledast)$ with L\'{e}vy pair $(\gamma,\sigma)$. 
Suppose $\{\mu_{n}\}_{n \in  \N} \subset \mathcal{ID}(\circledast)$ is a sequence of infinitely divisible probability measures. 
If $\sigma$ is supported over $m$ points, 
then the condition $\kappa_{n}[\mu_{n}] \rightarrow\kappa_{n}[\mu_{\infty}]$ for $n=1,\ldots,2m+2$ implies the convergence in distribution of $\mu_n$ towards $\mu_{\infty}$. 
\end{Theorem}

The explicit bijective correspondence between L\'{e}vy pairs  and $\circledast$-infinitely divisible distributions  can be found in \cite[Section 4]{MR3183991}. 
For our purposes, it is enough to recall that each $\mu_{\infty} \in  \mathcal{ID}(\circledast) $  determines a unique pair $(\gamma,\sigma)$, called \emph{L\'{e}vy pair}, where $\gamma$ is a real number and $\sigma$ is a finite measure over $\R$, and vice versa. 
Moreover, under this bijection, probability distributions that are not necessarily $\circledast$-Gaussian are obtained from L\'{e}vy pairs $(\gamma,\sigma)$ where $\sigma$ is supported on a finite set.

\begin{Notation}%\label{def:compound_poisson}
Let $\mathcal{M}(\R)$ denote the set of finite measures over $\R$. 
For every non-null measure $\Poiss \in \Mca(\R)$, we let $\tau_{\Poiss}$ denote the compound $\circledast$-Poisson distribution with rate $\Poiss[\R]$ and jump distribution $\Poiss/\Poiss[\R]$. 
The probability measure $\tau_\Poiss$ then satisfies 
\begin{align}\label{eqn:def_tau_nu}
			\tau_{\Poiss} = \lim_{n \to \infty} {\textstyle \left(\left(1-\frac{\Poiss[\R]}{n}\right)\delta_{0}+\frac{1}{n}\Poiss\right)^{\circledast n}  }
			\qquad \text{and} \qquad 
			\kappa_{n}[\tau_{\Poiss}] = m_n[\Poiss].
\end{align}
Additionally, if  $\Poiss\in \Mca(\R)$ has finite first moment $m_1[\Poiss] <  \infty$, we let $\pi_{\Poiss}$ denote $ \tau_{\Poiss} \circledast \delta_{-m_1[\Poiss]}$.  
\end{Notation}

%\jcomment{Decir algo del caso monótono, esencialmente hay que tomar en cuenta que la convolución monótona no es conmutativa.} 

Note that $g_{\sigma^2}$, $\tau_{\Poiss}$, and $\delta_\alpha$ are  infinitely divisible, and so is $g_{\sigma^2}  \circledast \pi_{\Poiss} \circledast \delta_a$. Our main contribution is a cumulant-type upper bound for the (non-commutative) $1$-Wasserstein distance (see Section \ref{section:wasserstein}) between $\circledast$-infinitely divisible distributions of the form 
%
%\begin{align} g_{\sigma^2}  \circledast \tau_{\Poiss} \circledast \delta_a \qquad \text{and} \qquad g_{\zeta^2}  \circledast \tau_{\Delta} \circledast \delta_b \end{align}
%
\begin{align}
g_{\sigma^2}  \circledast \pi_{\Poiss} \circledast \delta_a
\qquad \text{and} \qquad 
g_{\zeta^2}  \circledast \pi_{\Delta} \circledast \delta_b
\end{align}
with $\Poiss[\R \backslash \{0\}] > 0$ and $\Delta = \sum_{n=1}^{m} \lambda_n \delta_{x_n}$ a positive discrete measure with finite support contained in $\R \backslash \{0\}$. 
This allows us to compare distributions that are a three-fold convolution of a centered $\circledast$-Gaussian, a compound $\circledast$-Poisson, and a point Dirac measure (sometimes referred as a \emph{drift}). 
%Notice that $g_{\sigma^2}$ and  $g_{\zeta^2}$ are centered $\circledast$-Gaussian, $\tau_{\Poiss}$ and $\tau_{\Delta}$ are compound $\circledast$-Poisson, and $\delta_a$ and $\delta_b$ are Dirac measures (sometimes referred as a \emph{drifts}). 

\begin{Theorem}\label{thm:main_theorm}
Let $\circledast \in \{ \ast, \boxplus ,\uplus\}$. 
Take $\mu = g_{\sigma^2}  \circledast \pi_{\Poiss} \circledast \delta_a$ and $\nu = g_{\zeta^2}  \circledast\pi_{\Delta} \circledast \delta_b$ with $\Poiss$ and $\Delta$ as above. 
If $0\leq\zeta^2\leq \sigma^2$ and $\mu$ has finite moments of order $1,\ldots,2m+4$, 
then there exists a constant  $C(\Delta,m_2[\mu])>0$, that depends only on $\Delta$ and the second moment of $\mu$, such that 
\begin{align}\label{eqn:main_cumulant_inequality}
\freeWass{1} (\mu,\nu)  
\leq  
{\color{black} C(\Delta,m_2[\mu])} \cdot \left( \sum_{n=1}^{2m+4} \Big| \kappa^{\circledast}_n[ \mu ]-\kappa^{\circledast}_n[ \nu ] \Big| \right)^{ 1/2} 
\end{align}
where $\freeWass{1} (\cdot , \cdot)$ denotes the non-commutative $1$-Wasserstein distance, provided that the sum inside the parenthesis is smaller or equal than $1$. 
\end{Theorem}

\begin{Remark}
For $ \Poiss , \Delta \in \Mca(\R)$ with $\Delta = \sum_{n=1}^{m} \lambda_n \delta_{x_n}$, the constant $C(\Delta,m_2[\mu])$ %appearing in \eqref{eqn:main_cumulant_inequality} 
from Theorem \ref{thm:main_theorm} can be explicitly taken as
\begin{align}\label{eqn:constant_main_theorem}
C(\Delta,m_2[\mu])  =
1 + (M_\Delta +
2 m_2[\mu]^{1/2} ) \left( \frac{ 2 \, U_\Delta}{  \radio_\Delta } \right)^{m}
\end{align}
where $M_\Delta$ denotes the term 
\begin{equation*}
M_\Delta
= 
3 
+
m \, U_\Delta  \, N_\Delta 
+
2U_\Delta  \, m \, (m+1) \, N^2_\Delta \, 
 \left( \frac{ 2 \, U_\Delta}{  \radio_\Delta } \right)^{m}  
\end{equation*}
with $U_{\Delta}, r_{\Delta}$, and $N_{\Delta}$
 given by \eqref{def:U_Delta_L_Delta}, \eqref{def:D_Delta_d_Delta}, and  \eqref{eq:Ndeltadef}, respectively. We make no claim on the sharpness of the constant \eqref{eqn:constant_main_theorem}.
\end{Remark}
It is worth emphasizing that our proof of Theorem \ref{thm:main_theorm} relies on fundamental features of $\circledast$-cumulants and elementary computations. 
%
%{\color{}}
In particular, we use the fact that $\circledast \in \{ \ast, \boxplus ,\uplus\}$ is commutative, meaning that $\mu \circledast \nu$ and $\nu \circledast \mu$ yield the same distribution.  
We expect to address the monotone case $\rhd$, which is non-commutative, elsewhere. 
Moreover, our proof exploits the following two properties of the free $1$-Wasserstein distance:
\begin{align}\label{eqn:free_Wp_subadditive}
	\freeWass{1} ( \mu_1 \circledast \mu_2 , {\nu}_1 \circledast {\nu}_2) \leq
	\freeWass{1} ( \mu_1 , {\nu}_1 ) + 
	\freeWass{1} ( \mu_2 , {\nu}_2), 
\end{align}
for any probability distributions $\mu_1,\mu_2,\nu_1,\nu_2\in \Pc(\R)$, and
\begin{align}\label{eqn:Wp_pi_nu}
\freeWass{1} ( \tau_{\Poiss_1} \circledast \delta_{-m_1[\Poiss_1]} , \tau_{\Poiss_2} \circledast \delta_{-m_1[\Poiss_2]} ) \leq
\classWass{1} ( \Poiss_1 , \Poiss_2 )
     +|m_1[\Poiss_1]-m_1[\Poiss_2]|,
\end{align}
for any finite measures $\Poiss_1 , \Poiss_2 \in \mathcal{M}(\R)$ with  $\Poiss_1[\R]=\Poiss_2[\R]$ and finite first moment. 
We leave the proofs of \eqref{eqn:free_Wp_subadditive} and \eqref{eqn:Wp_pi_nu}, and more details, for Section \ref{section:wasserstein}. 
%

%\subsection{Manuscript's organization}

%removecomment\fcomm{1. Write manuscript's organization} 

The rest of this manuscript is organized as follows. 
In Section \ref{section:preliminaries}, we provide some preliminaries on non-commutative probability, this includes basic properties of  $\ast$-probability and $C^{\ast}$-probability spaces as well as the concrete definition of $\circledast$-cumulants and the heuristic  %\ref{eqn:cpoisson_in_cumulants} 
that Definition \ref{def:compound_Poisson} yields genuine Poisson distributions. 
% 
%$\circledast$-infinitely divisible distributions. 
%removecomment \ \\ \ \\
%
In Section \ref{section:wasserstein}, we recall the definition of free Wasserstein distance of Biane and Voiculescu and give proofs for the relations \eqref{eqn:Wp_pi_nu} and \eqref{eqn:free_Wp_subadditive} above.  
%removecomment \ \\ \ \\
% 
Finally, in Section \ref{section:main_theorem}, we prove Theorem \ref{thm:main_theorm}, our main result, by splitting its proof into various steps. % for clarity. % and readability. 

%--------------------------------------------------------------------------------------%
%--------------------------------------------------------------------------------------%
%                             SECTION:FREE_PRELIMINARIES                               %
%--------------------------------------------------------------------------------------%
%--------------------------------------------------------------------------------------%

%removecomment \newpage 

\section{Preliminaries}\label{section:preliminaries}

\subsection{Non-commutative probability}\label{subsection:ncp}
A \emph{non-commutative probability space} is a pair $(\Ac,\varphi)$ where $\Ac$ is an algebra over $\C$ with multiplicative identity $1_{\Ac} \in \Ac$ and $\varphi: \Ac \rightarrow \C$ is a linear functional such that $\varphi(1_{\Ac})=1$. 
A non-commutative probability space $(\Ac,\varphi)$ is said to be \emph{tracial} if $\varphi(ab)=\varphi(ba)$ for all $a,b \in \Ac$. 
This is however a purely algebraic setting with no analytic structure. 
For this reason, when investigating analytic properties in non-commutative probability, it is customary to work with $\ast$-probability and $C^{\ast}$-probability spaces instead. 

\begin{Definition}\label{def:cstar_probability_spaces}
A non-commutative probability space $(\Ac,\varphi)$ is said to be a $\ast$\emph{-probability space} if it is endowed with an anti-linear operation $\ast: \Ac \rightarrow \Ac$ such that $(a^*)^* = a$, $(ab)^*  = b^* a^*$, and $\varphi(a^*a) \geq 0$ for every $a,b \in \Ac$. 
If, in addition, $\Ac$ is a $C^{\ast}$-algebra, we refer to the pair  $(\Ac,\varphi)$ as a $C^{\ast}$-\emph{probability space}. 
\end{Definition}

The book \cite{MR2266879} is a standard reference on  $\ast$-probability and $C^{\ast}$-probability spaces, covering general background and main features. 
Here we recollect some properties relevant to this manuscript. 

\begin{Example}\label{exa:star_probability_space}
Let $(\Omega, \Fc, \mu)$ be a classical probability space, i.e.,  $\Omega$ is a non-empty set, $\Fc$ is a $\sigma$-algebra of subsets of $\Omega$, and $\mu:\Fc \to [0,1]$ is a probability measure. 
A canonical example of a $\ast$-probability space  $(\Ac,\varphi)$ is the algebra of all classical random variables on $(\Omega, \Fc, \mu)$ with finite absolute moments of all orders. 
Namely, $\Ac = \cap_{ p \geq 1 } L^p(\Omega, \Fc, \mu)$ and  
$\varphi(a) = \int_{\Omega} a(\omega) \, \mu(d\omega)$ for each $a \in \Ac$.
\end{Example}

Recall that a complex algebra $\Ac$ is a $C^{\ast}$\emph{-algebra} if 
it is a Banach algebra with respect to some norm $\norm{\, \cdot \,}$ such that $\norm{ab} \leq \norm{a} \norm{b}$ for all $a,b \in \Ac$
and it is also endowed with an anti-linear operation $\ast: \Ac \rightarrow \Ac$ such that $(a^*)^* = a$, $(ab)^*  = b^* a^*$, and $\norm{aa^*} = \norm{a}^2$ for all $a,b \in \Ac$. 

\begin{Example}
A fundamental example of a $C^{\ast}$-probability space $(\Ac,\varphi)$ arises from Hilbert spaces $\Hc$ and unit vectors $\xi \in \Hc$ as follows.    
Let $\Ac$ be the algebra of bounded linear operators $B(\Hc)$ with the usual operator norm and 
take $ \varphi: \Ac \to \C$ given by $\varphi(T) = \langle  T \xi, \xi \rangle $ for each $T \in \Ac$. 
The corresponding $\ast$-operation is given by the adjoint, i.e., for each $T \in \Ac$, the operator $T^{\ast} \in \Ac$ is uniquely determined by the relation $\langle T \eta_1 , \eta_2 \rangle =\langle \eta_1 , T^{\ast} \eta_2 \rangle$ for all $\eta_1, \eta_2 \in \Hc$. 
\end{Example}

A \emph{representation} of a $\ast$-probability space  $(\Ac,\varphi)$ is a triple $(\Hc,\pi,\xi)$ where 
$\Hc$ is a Hilbert space, 
$\pi:\Ac \rightarrow B(\Hc)$ is a unital $\ast$-homomorphism, and 
$\xi \in \Hc$ is a unit vector such that $\varphi(a) = \langle \pi(a) \xi , \xi \rangle$ for every $a \in \Ac$. 

\begin{Remark}\label{rmk:representation_of_L_-infinity}
The  $\ast$-probability space from Example \ref{exa:star_probability_space} has a natural representation on the Hilbert space $\Hc = \mathcal{L}^2(\Omega, \Fc, P)$ of all square-integrable functions.
One simply takes the vector $\xi \in \Hc$ as the function $\xi(\omega)=1$ for all $\omega \in \Omega$ and $\pi(a)(T) = aT$ for all $a \in \Ac$ and $T \in \Hc$. 
Notice that this representation implies that, for any probability distribution $\mu \in \mathcal{P}_{\infty}(\R)$ with compact support, there always exists a $C^{\ast}$-probability space with a random variable whose  analytical distribution in the sense of \eqref{eqn:phi-moments-mu} is exactly $\mu$.
%
%Indeed, given $\mu \in \mathcal{P}_{\infty}(\R)$, if we take $\Omega = \R$, $\mathcal{F} = \mathcal{B}(\R)$, $P = \mu$, and $a = id_{\R}$, then $\pi(a) \in B(\mathcal{H})$ has analytical distribution $\mu$.
%
\end{Remark}

An element in a $C^\ast$-algebra $p \in \Ac$ is \emph{positive} if there exists $a \in \Ac$  such that $p = a^* a$. 
Moreover, the functional calculus for $C^\ast$-algebras, see \cite[Theorem 3.1]{MR2266879}, justifies the application of continuous functions to normal ---including positive--- elements. % in a $C^\ast$-algebra. 
In particular,  this implies that for any integer $k \geq 1$ and any $a \in \Ac$, there exists a positive $b \in \Ac$ that we identify with $(a^* a)^{k/2}$, in the sense that $b^2=(a^* a)^{k}$; 
in addition, if $(\Ac,\varphi)$ is a $C^{\ast}$-probability space, then $\varphi(b) \geq 0$, yielding the following definition.

\begin{Definition}\label{def:nc_pnorm}
Let $(\Ac,\varphi)$ be a $C^{\ast}$-probability space. 
The \emph{non-commutative $p$-norm} $\abs{\, \cdot \,}_p :  \Ac \rightarrow  \R_{\geq 0}$ is defined for each $a \in  \Ac$  by 
%the norm $\abs{\, \cdot \,} :  \Ac \rightarrow \Ac $ given for each $a \in  \Ac$ by 
	%
	\begin{align}\label{eqn:noncommutaive_pnorm}
		\abs{\, a \,}_p : =
	%	\norm{ \varphi[ (a^* a)^{p/2} ]}^{1/p}
	 	\sqrt[p]{  \, \varphi[ (a^* a)^{p/2} ] \ \ } .
	\end{align}
	%
%where $\norm{\, \cdot \,}$ denotes the usual norm on $\C$. 
%
%
In particular, for $p=1$, we have \begin{align}\label{eqn:noncommutaive_1norm} 	\abs{\, a \, }_1 = 	\varphi[ (a^* a)^{1/2} ] . \end{align}

\end{Definition}
%

%removecomment \jcomment{\ \\ Here we need to cover $C^{\ast}$-probability spaces and the non-commutative $p$-norm \eqref{eqn:noncommutaive_pnorm} }

%removecomment \fcomm{2. Finish writing preliminaries on non-commutative probability spaces}

\subsection{Notions of independence}

Let $(\mathcal{A}_i)_{i \in I}$ be a family of subalgebras of $\Ac$.
Suppose $a_1 \in \Ac_{i_1},a_2 \in \Ac_{i_2}, \ldots, a_n\in \Ac_{i_n}$ for some indexes $i_1,\ldots,i_n \in I$. 
We say that $a_1 a_2 \cdots a_n$ is \emph{an alternating product of elements of} $(\mathcal{A}_i)_{i \in I}$ if the indexes $i_k$ satisfy that $i_1 \neq i_2$, $i_2 \neq i_3$, $\ldots$, $i_{n-1} \neq i_n$.

\begin{Definition}\label{dfn:boolean_tensor_monotone}
	Let $(\Ac,\varphi)$ be a non-commutative probability space. A family $(\mathcal{A}_i)_{i \in I}$  of subalgebras of $\mathcal{A}$ is said to be   
	\begin{enumerate}
		\item[(\textbf{B})] \emph{Boolean independent} if for any alternating product $a_1 a_2 \cdots a_n$ of elements of  $(\mathcal{A}_i)_{i \in I}$ we have
		\[\varphi(a_1 a_2 \cdots a_n)=\varphi(a_1)\varphi(a_2)\cdots \varphi(a_n) \ .\] 
		\item[(\textbf{F})] \emph{free independent} if for any alternating product $a_1 a_2 \cdots a_n$ of elements of  $(\mathcal{A}_i)_{i \in I}$  we have 
		\begin{align*}
\varphi(a_1 a_2 \cdots a_n) = 0
\end{align*}
		whenever $\varphi(a_k) = 0$ for $k=1,2,\ldots,n$. 

		\item[(\textbf{T})] \emph{tensor independent} if for any (not necessarily alternating) product $a_1 a_2 \cdots a_n$ of elements of  $(\mathcal{A}_i)_{i \in I}$ we have 
		\[
		\varphi(a_1 a_2 \cdots a_n) \ \ = \prod_{B \in \ker[{\bm i}]} \varphi \left( \OrdAsProduct{a}{k}{B} \right) 
		\]
		where $\overset{\longrightarrow}{\prod}_{k \in B} a_k := a_{k_1} a_{k_2} \cdots a_{k_r}$
		%$\overset{\longrightarrow}{\underset{k \in B}{\prod}} a_k := a_{k_1} a_{k_2} \cdots a_{k_r}$ 
	provided $B = \{k_1 < k_2 < \cdots < k_r \} \subset [m] := \{1,2,\ldots,m\} $ and $\ker[{\bm i}]$ with ${\bm i}=(i_1,i_2,\ldots,i_n)$ denotes the set of all non-empty sets of the form $ \{k \in [m] : i_k = i\}$ with $i \in I$. 
		Essentially, the set $\ker[{\bm i}]$ is the partition of $[m]$ that keeps track of all variables that belong to the same algebra.  
	\end{enumerate}
\end{Definition}

%\begin{Example}
As an example for the notation $\ker[{\bm i}]$, let us consider ${\bm i}=(1,3,1,7,7,1)=(i_1,i_2,i_3,i_4,i_5,i_6)$. 
In this case, we have that $\ker[{\bm i}]$ denotes the partition $\{ \{ 1,3,6\}, \{2\}, \{4,5\}\}$ since $i_1 = i_3 = i_6$, $i_4=i_5$, and $i_1,i_2,i_4$ are all distinct values; 
additionally, the definition for tensor independence yields that $\varphi(a_1 a_2 \cdots a_6)$ would factor as $\varphi(a_1 a_3 a_6) \varphi(a_2) \varphi(a_4 a_5)$.
%\end{Example}

\begin{Remark}\label{rem:product_space}
Let $\circledast \in \{ \ast, \boxplus ,\uplus\}$. 
For any family of $C^{\ast}$-probability spaces $(\mathcal{A}_i,\varphi_i)_{i \in I}$, there always exist another $C^{\ast}$-probability space $(\mathcal{A},\varphi)$ together with injective homomorphisms $\Psi_i:\mathcal{A}_i \to \mathcal{A}$ satisfying $\varphi_i(a)= \varphi \circ \Psi_i(a)$ for every $a \in \mathcal{A}_i$ and $i \in I$ such that $(\Psi_i(\mathcal{A}_i))_{i \in I}$ are $\circledast$-independent in $(\mathcal{A},\varphi)$.
Indeed, we can take $\mathcal{A}$ as the free product $C^{\ast}$-algebra generated by the family $(\mathcal{A}_i)_{i \in I}$, see \cite[Lecture 7]{MR2266879}, and then $\varphi$ is uniquely determined on generators by Definition \ref{dfn:boolean_tensor_monotone} and the condition that $\varphi_i(a)= \varphi \circ \Psi_i(a)$ for every $a \in \mathcal{A}_i$ and $i \in I$;
moreover, if every $(\mathcal{A}_i,\varphi_i)$ is tracial, so is $(\mathcal{A},\varphi)$. 
\end{Remark}

%removecomment\newpage 
\subsection{Cumulants}\label{subsection:cumulants}
To give a concrete definition of $\circledast$-cumulants, 
% $\kappa_n^{\circledast}[\mu]$ of a probability measure $\mu \in \mathcal{P}_{\infty}(\R) $, 
we first need to introduce the notions of partition, non-crossing partition, and interval partition. 

Let $[n]$ denote the set $\{1,2,\ldots,n\}$ for each integer $n \geq 1$. 
A \emph{partition} $\pi = \{ B_1, B_2, \ldots, B_r\}$ of $[n]$ is a set of non-empty and pair-wise disjoint subsets of $[n]$ whose union is $[n]$,
i.e., $B \subset [n]$ and $B \neq \emptyset$  for every $B \in \pi$, $B\cap B' \neq \emptyset$ implies $B=B'$ for all $B,B'\in \pi$, and $\cup_{B \in \pi } B = [n]$. 
The elements of a partition are referred to as \emph{blocks}. 
The set of all partitions of $[n]$ is denoted by $P(n)$.
%
%
%
%
%

%and total number of blocks in a partition $\pi$ is denoted by $\#(\pi)$. 
%
Given two distinct blocks  $B=\{ b_1 < b_2 < \cdots < b_r\}$ and $C=\{ c_1 < c_2 < \cdots < c_s\}$ from a partition $\pi \in P(n)$,  we say $B$ and $C$ \emph{cross each other} if there exist $b_i,b_j \in B$ and $c_k, c_\ell \in C$ such that $b_i < c_k < b_j < c_\ell$.
A partition $\pi \in P(n)$ is called \emph{non-crossing} if there is no pair of distinct blocks  $B,C \in \pi$ that cross each other. 
The set of all non-crossing partitions of $[n]$ is denoted by $NC(n)$. 

A partition $\pi \in P(n)$ is called \emph{interval} if every block of $\pi$ is an interval of consecutive numbers of $[n]$, i.e., if $B \in \pi$, then $B=\{\ell, \ell+1, \ldots , m \}$ for some $\ell,m \in [n]$ with $\ell \leq m$. 
Each interval partition is non-crossing as well. 
The set of all interval partitions of $[n]$ is denoted by $I(n)$. 
%
%
%
%

%\fcomm{Write about various types of partition sets $P(n)$, $NC(n)$, and $I(n)$, as well as the Mobius inversion formulas that determine moment-cumulant and cumulant-moment formulas}

%For instance, the sets $\pi_{1} = \{\{1,3\},\{2,4,5,6\}\}$, $\pi_{2} = \{\{1,3,6\},\{2\},\{4,5\}\}$, and $\pi_{3} = \{\{1,3\}  $ $\{4,6\},\{2,5\}\}$ are all partitions of $\{1,2,3,4,5,6\}$. 
%
\begin{Example}
The sets $\pi_{1} = \{\{1,3\},\{2,4,5,6\}\}$, $\pi_{2} = \{\{1,3,6\},\{2\},\{4,5\}\}$, and $\pi_{3} = \{\{1\}, \{2,3,4\},\{5,6\}\}$ are all partitions of the set $\{1,2,3,4,5,6\}$, however, only $\pi_2$ and $\pi_3$ are non-crossing. 
Additionally, $\pi_3$ is an interval partition but $\pi_2$ is not. 
\end{Example}

\begin{Notation} 
Let $\circledast \in \{ \ast, \boxplus ,\uplus\}$.  
For each integer $n \geq 1$, we let 
	\begin{align}
		P_{\circledast}(n) = 
		\left\{\begin{array}{rl}
			P(n), & \text{if } \circledast = \ast, \\
			NC(n), & \text{if } \circledast = \boxplus, \\
			I(n),	& \text{if } \circledast = \uplus.
	 \end{array} \right.
	\end{align}
\end{Notation}

\begin{Definition}\label{def:cumulants}
	For a probability measure $\mu \in \mathcal{P}_{\infty}(\R) $, its $\circledast$-cumulants $\kappa_n^{\circledast}[\mu]$ are defined recursively through the moment-cumulant formula 
\begin{align}\label{eqn:moments_cumulants}
	m_n[\mu] \ \ = \sum_{\pi \in P_{\circledast}(n)} \kappa^{\circledast}_{\pi} [\mu]  \quad \text{with} \quad 
	\kappa^{\circledast}_{\pi} [\mu] : = \prod_{V \in \pi} \kappa^{\circledast}_{\abs{V}} [\mu] \ .
\end{align} 
\end{Definition}

\begin{Remark}
The set $P_{\circledast}(n)$ becomes a lattice when endowed with the partial order given by refinement. 
Namely, given two partitions $ \pi, \theta \in P_{\circledast}(n)$, we write $\pi \leq \theta$, and say that $\pi$ is a \emph{refinement} of $\theta$, if every block of $\pi$ is contained in some block of $\theta$.
Under this partial order, the partitions $1_n:=\{ \{1,2,\ldots,n\}\}$ and $0_n:=\{\{1\},\{2\},\ldots,\{n\}\}$ are the largest and smallest element, respectively. 
The fact that \eqref{eqn:moments_cumulants} yields a well-defined recursion formula follows from 
\begin{align*}
    m_n[\mu] \ \ =  \ \ \kappa_n[\mu] \ \ + \sum_{\substack{\pi \in P_{\circledast}(n) \\ \pi \neq 1_n }} \kappa^{\circledast}_{\pi}[\mu]
\end{align*}
where the sum in the right hand side depends only on $\kappa_1[\mu],\kappa_2[\mu],\ldots,\kappa_{n-1}[\mu]$. 
Moreover, see \cite[Proposition 10.6]{MR2266879}, there exists a function $\mu:P_{\circledast}(n) \times P_{\circledast}(n) \to \Z$, commonly known as \emph{M\"{o}bius function}, that allows us to invert the moment-cumulant formula \eqref{eqn:moments_cumulants}.
Concretely, the $\circledast$-cumulants can be equivalently defined through the cumulant-moment formula 
\begin{align}\label{eqn:cumulants_moments}
		\kappa^{\circledast}_n[\mu] \ \ = \sum_{\pi \in P_{\circledast}(n)} m_{\pi} [\mu]  \cdot \mu_{\circledast}(\pi,1_n) \ . 
\end{align}
\end{Remark}

The $\circledast$-cumulants are also uniquely determined axiomatically, see \cite[Section 3]{MR2884229}, by the following three properties:  \begin{enumerate}
		\item Additivity
		\begin{align*}
			\kappa^{\circledast}_n[\mu_1 \circledast \cdots \circledast \mu_r ]
			=
			\kappa^{\circledast}_n[\mu_1 ]
			+ \cdots +
			\kappa^{\circledast}_n[\mu_r ] 
		\end{align*}
        
		\item Homogeneity
		\begin{align*}
			\kappa^{\circledast}_n[{}_{_\lambda}\mu] %[\mu_{\lambda}  ]
			=
			\lambda^n \, \kappa^{\circledast}_n[\mu]
		\end{align*}

        \item Polynomiality
        \begin{align*}
			\kappa^{\circledast}_n[\mu ] 
			=
			m_n[\mu]
            +
            q_n (\kappa^{\circledast}_1[\mu], \kappa^{\circledast}_2[\mu], \cdots, \kappa^{\circledast}_{n-1}[\mu])
		\end{align*}
        for a polynomial $q_n$ that depends only on $\circledast$ and $n$

\end{enumerate}

%removecomment \fcomm{3. Finish writing preliminaries on cumulants} 

%removecomment \newpage 

%removecomment \fcomm{4. Finish writing preliminaries on $\infty$-divisibility} 

\subsection{Infinitely divisible distributions}\label{subsection:infinitely_divisible}
%removecomment\ \\ \ \\ 
%removecomment \jcomment{La demostración de la ecuación \eqref{approxcomppoisson} es un poco más fina, requiere de la función de Mobius. \\ \ \\ Things to be included \begin{enumerate} 	\item Law of rare events. Mentioned in the introduction. \end{enumerate}}
%
As with Gaussian distributions, one can justify that \eqref{eqn:def_tau_nu} does give genuine Poisson distributions, in the sense that it characterizes the measures arising from laws of rare events. 
In this case, one considers Equation \eqref{eqn:convergence_in_cumulants} with $ \mu_n  = ( {\textstyle ( 	1-\frac{\lambda}{n} ) \,	\delta_{0} 	+\frac{\lambda}{n} \, \Poiss } )$. 	
Let $\mathcal{M}_{\infty}(\R) $ be the set of finite measures over $\R$ with finite moments of all orders.
%{\color{red}Let $\mathcal{M}_K(\R^*)$ denote the set of finite measures on $\R^*=\R\setminus\{0\}$ with compact support.}
%
%
%
%\jcomment{Revisar la notación. Creo que no se ha usado $\mathcal{M}_K(\R^{*})$ antes.}

\begin{Lemma} Suppose $\Poiss \in \mathcal{M}_{\infty}(\R)$. Then, the sequence
	\begin{align}\label{approxcomppoisson}
    \textstyle
		\left(\left(1-\frac{\Poiss[\R]}{n}\right)\delta_{0}+\frac{1}{n}\Poiss\right)^{\circledast n}
	\end{align}
	converges weakly as $n \to \infty$ towards a limiting distribution $\tau_{\Poiss}$ characterized by the relation 
	\begin{align*}
		\kappa_\ell^{\circledast}[\tau_{\Poiss}]
		&=m_\ell[\Poiss].
	\end{align*}
	The distribution $\tau_{\Poiss}$ is referred to as the \emph{compound Poisson distribution} with Levy measure $\Poiss$.
\end{Lemma}
\begin{proof}
For each integer $n \geq 1$, let  
	\begin{align*}
\textstyle
		\mu_n
		 =\left(1-\frac{\Poiss[\R]}{n}\right)\delta_{0}+\frac{1}{n}\Poiss 
\quad \text{and} \quad 
\rho_n = 	\underbrace{\mu_n \circledast \cdots \circledast \mu_n}_{n \text{ times}}.
	\end{align*}
Since $\circledast$-cumulants linearize convolution, we have that $ \kappa^{\circledast}_{\ell} [ \rho_n  ] = n \cdot \kappa^{\circledast}_{\ell} [  \mu_n ] $ for any integer $\ell \geq 0$. 
Thus, since convergence in distribution is equivalent to convergence in moments in the space of measures characterized by its moments, it is enough to show that 
\begin{align}\label{eqn:poisson_cumulant_limit}
m_{\ell} [\Poiss] = \lim_{n \to \infty} \, n \cdot \kappa^{\circledast}_{\ell} [  \mu_n ].  %\lim_{n \to \infty} \, \kappa^{\circledast}_{\ell} [ \rho_n  ]. 
\end{align}
To this end, note that the $\ell$-th moment of $\mu_n$ satisfies 
$m_\ell[\mu_n] =  \frac{1}{n} m_\ell[\Poiss]$. 
So, the cumulant-moment formula \eqref{eqn:cumulants_moments} yields 
\begin{align}\label{eqn:moment-cumulant-poisson}
	\kappa_\ell^{\circledast}[\mu_n] 
\ \ = 
		\sum_{\pi \in P_{\circledast}( \ell )} 
			m_{\pi} [\mu_n]  \cdot 
			\mu_{\circledast}(\pi,1_\ell) 
 \ \ = 
	\sum_{\pi \in P_{\circledast}( \ell )}
			 \frac{1}{n^{\#(\pi)}} \cdot  m_{\pi} [\Poiss]  \cdot \mu_{\circledast}(\pi,1_\ell) .
\end{align}
Consequently, when multiplying the left-hand side of \eqref{eqn:moment-cumulant-poisson} by $n$, and then letting $n$ go to infinity, every term in the sum corresponding to a partition $\pi$ with more than one block will vanish. 
Since $1_{\ell}$ is the only partition in $P_{\circledast}(\ell)$ with one block, $\mu_{\circledast}(1_\ell,1_\ell)=1$, and $m_{1_{\ell}}[\Poiss] = m_{\ell}[\Poiss]$, we obtain that \eqref{eqn:poisson_cumulant_limit} holds.
\end{proof}

\begin{Remark}
The $\circledast$-cumulants of a infinitely divisible distribution of the form $ {g}_{\sigma^2}\circledast\pi_{\Poiss} \circledast \delta_{a}$ are given by 
\begin{align}\label{eq:cumulantsidd}
	\kappa_n^{\circledast}[ {g}_{\sigma^2}\circledast\pi_{\Poiss} \circledast \delta_{a}] 
	= 
\left\{\begin{array}{rl}
		a, &  \text{if } n = 1, \\
		\sigma^2 + m_2[\Poiss], &  \text{if } n = 2, \\
		m_n[\Poiss], &  \text{if } n\geq 3. \\
		\end{array} \right.
\end{align}

\end{Remark}

%removecomment \begin{align}\kappa_p[\mu]&=\Indi{\{p=1\}}b+\Indi{\{p=2\}}\left(\sigma^{2}+m_2[\nu]\right)+\Indi{\{p\geq 3\}}m_p[\nu] \end{align}

%--------------------------------------------------------------------------------------%
%--------------------------------------------------------------------------------------%
%                             SECTION:FREE_WASSERSTEIN                                 %
%--------------------------------------------------------------------------------------%
%--------------------------------------------------------------------------------------%

%removecomment \newpage 

\section{Free Wasserstein distance}\label{section:wasserstein}

The distance on $\mathcal{P}(\R)$ that we consider in this manuscript is free $p$-Wasserstein distance introduced by Biane and Voiculescu in \cite{MR1878316}.
\begin{Definition}
For $\mu_1,\mu_2 \in \mathcal{P}(\R)$, we let $\Gamma_{\!\!\textrm{nc}}(\mu_1,\mu_2)$ denote the set of all tracial $C^{*}$-probability spaces $(\Ac,\varphi)$ containing variables $a_1=a_1^*$ and $a_2=a_2^*$ with analytical distributions $\mu_1$ and $\mu_2$, respectively. 
The \emph{free $p$-Wasserstein distance} between $\mu_1$ and $\mu_2$ is defined as
\begin{align}
	\freeWass{p} (\mu_1,\mu_2) \quad = \inf_{ (\Ac,\varphi) \in \Gamma_{\!\!\textrm{nc}}(\mu_1,\mu_2) } 
			\left\{ \abs{\, a_1-a_2 \,}_p  \, \left| \begin{array}{c}
			a_1=a_1^*,a_2=a_2^* \in \Ac \text{ with } \\  \varphi(a^n_i) = m_n[\mu_i] 
			\text{ for } n \geq 1 \text{ and } i=1,2			
	\end{array} \right.			
			 \right\}	
\end{align}
where $\abs{\, \cdot \,}_p $ is the the $p$-norm in a tracial $C^{\ast}$-probability spaces, see Definition \ref{def:nc_pnorm}. 
\end{Definition}

\begin{Remark}\label{rem:free_Wass_attained}
The free $p$-Wasserstein distance is always attained. 
Concretely, in \cite[Proposition 1.4]{MR1878316}, it is proved that given any measures $\mu_1,\mu_2 \in \mathcal{P}(\R)$ there always exists a tracial  $C^{*}$-probability space $(\Ac,\varphi)$ and self-adjoint random variables $a_1, a_2 \in(\Ac,\varphi) $ with analytical distributions $\mu_1$ and $\mu_2$, respectively, such that 
\begin{align*}
\freeWass{p}(\mu_1,\mu_2) = \abs{a_1 - a_2 }_p.
\end{align*}
\end{Remark}

The free $p$-Wasserstein distance is  bounded above by its classical counterpart.  
Recall that the \emph{classical $p$-Wasserstein distance} between two finite measures $\nu_1,\nu_2 \in \mathcal{M}(\R)$ is defined as 
\begin{align}\label{Wasserstein:frevsclass}
	\classWass{p}( \nu_1, \nu_2 ) \ \, : = \inf_{\gamma \in \Gamma(\nu_1,\nu_2)} 
	\left( \int_{\R^2} \abs{x-y}^p \, d \gamma(x,y) \right)^{1/p}
\end{align}
where $\Gamma( \nu_1, \nu_2)$ is the set of all couplings of $ \nu_1$ and $ \nu_2$. 

\begin{Lemma}\label{lemma:convolution_wasserstein-inequality}
Let $\circledast\in\{\ast,\boxplus,\uplus\}$.  If $\mu_1$, $\mu_2$, $\nu_1$, and $\nu_2$ are probability measures on $\R$, then the free Wasserstein distance $\freeWass{p}$ satisfies 
\begin{align}\label{eqn:convolution_wasserstein-inequality}
    \freeWass{p} ( \mu_1 \circledast \mu_2 , \nu_1 \circledast \nu_2  ) \leq  \freeWass{p} ( \mu_1 , \nu_1  ) +  \freeWass{p} ( \mu_2 ,  \nu_2  )
\end{align}
\end{Lemma}
\begin{proof}
There is a tracial $C^{\ast}$-probability space $(\Ac_1,\tau_1)$  and self-adjoint random variables $a_1,b_1 \in (\Ac_1,\tau_1)$ with analytical distributions $\mu_1,\nu_1$, respectively, such that 
\begin{align*}
    \abs{a_1 - b_1 }_p = \freeWass{p}(a_1,b_1) = \freeWass{p}(\mu_1,\nu_1) ,  
\end{align*}
see Remark \ref{rem:free_Wass_attained}.
Additionally, see Remark \ref{rmk:representation_of_L_-infinity}, there are tracial $C^{\ast}$-probability spaces $(\Ac_2,\tau_2)$ and $(\Ac_3,\tau_3)$ and self-adjoint random variables $a_2\in (\Ac_2,\tau_2)$ and $b_2 \in (\Ac_3,\tau_3)$ with analytical distributions $\mu_2$ and $\nu_2$, respectively. 
Consequently, see Remark \ref{rem:product_space}, there exists a product $C^{\ast}$-probability space $(\Ac,\tau)$ and self-adjoint random variables $x_1,x_2,y_1,y_2 \in (\Ac,\tau)$ such that $\{ x_1, y_1\}$, $\{ x_2 \}$, and $\{ y_2 \}$ are $\circledast$-independent, $\mu_j$ and  $\nu_j$ are the analytical distributions of $x_j$ and $y_j$, respectively, and 
\begin{align*}
    \abs{x_1 - y_1 }_p = \freeWass{p}(x_1,y_1) = \freeWass{p}(\mu_1,\nu_1) . 
\end{align*}
By $\circledast$-indepedence, $\mu_1 \circledast \mu_2$ and $\nu_1 \circledast \nu_2$ are the analytical distributions of $x_1 + x_2$ and $y_1 + y_2$, respectively. 
Thus, the definition and triangle inequality for the non-commutative Wasserstein distance $ \freeWass{p} $ imply that
\begin{align}\label{eqn:wasserstein-inequality_triangle}
    \freeWass{p} ( \mu_1 \circledast \mu_2 , \nu_1 \circledast \nu_2  ) 
   % =  \freeWass{p} ( x_1 + x_2 , y_1 + y_2  ) 
    \leq   
    \freeWass{p} ( x_1 + x_2 , x_2 + y_1  ) +   \freeWass{p} ( x_2 + y_1 , y_1 + y_2  )  .  
\end{align}
Now, there also exists a (possibly distinct) tracial $C^{\ast}$-probability space $(\tilde{\Ac},\tilde{\tau})$ and self-adjoint random variables $\tilde{x}_1,\tilde{x}_2,\tilde{y}_1,\tilde{y}_2 \in (\tilde{\Ac},\tilde{\tau})$ such that $\{ \tilde{x}_2, \tilde{y}_2\}$, $\{ \tilde{x}_1 \}$, and $\{ \tilde{y}_1 \}$ are $\circledast$-independent, $\mu_j$ and  $\nu_j$ are the analytical distributions of $\tilde{x}_j$ and $\tilde{y}_j$, respectively, and 
\begin{align*}
    \abs{\tilde{x}_2 - \tilde{y}_2 }_p = \freeWass{p}(\tilde{x}_2 , \tilde{y}_2) = \freeWass{p}(\mu_2,\nu_2) . 
\end{align*}
Note that $x_2 + y_1$ and $\tilde{x}_2+\tilde{y}_1$ have the same analytical distribution $\mu_2 \circledast \nu_1$, and similarly $y_1 + y_2$ and $\tilde{y}_1+\tilde{y}_2$ with respect to $\nu_1 \circledast \nu_2$, due to $\circledast$-indepedence.  
Hence, we have that  
\begin{align*}
     \freeWass{p} ( x_2 + y_1 , y_1 + y_2  ) 
     =
     \freeWass{p} ( \tilde{x}_2+\tilde{y}_1, \tilde{y}_1+\tilde{y}_2 ) 
     \leq 
     \abs{ (\tilde{x}_2+\tilde{y}_1) - (\tilde{y}_1+\tilde{y}_2) }_p 
      = \freeWass{p}(\mu_2,\nu_2).  
\end{align*}
Finally, since we also have 
\begin{align*}
    \freeWass{p} ( x_1 + x_2 , x_2 + y_1  ) \leq   \abs{ (x_1+x_2) - (x_2 + y_1) }_p  = \freeWass{p}(\mu_1,\nu_1), 
\end{align*}
the conlusion of the lemma, inequality \eqref{eqn:convolution_wasserstein-inequality}, readily follows from \eqref{eqn:wasserstein-inequality_triangle} and the last two inequalities. 

\end{proof}

\begin{Lemma}\label{lemma:convolution_inequality}
If $\Poiss_1$ and $\Poiss_2$ are finite measures on $\R$ with the same total mass $\Poiss_1[\R]=\Poiss_2[\R]$ and finite first moments $m_1[\Poiss_1]$ and $m_1[\Poiss_2]$, then 
\begin{align}\label{eqn:convolution_inequality}
     \freeWass{1} ( \pi_{\Poiss_1}  , \pi_{\Poiss_2}  )  \leq \classWass{1} ( \Poiss_1 , \Poiss_2 )
     +|m_1[\Poiss_1]-m_1[\Poiss_2]|
\end{align}
with $\freeWass{1}$ and $\classWass{1}$ the free and classical $1$-Wasserstein distance, respectively. 
\end{Lemma}

\begin{proof}
For any integers $1\leq i \leq n$, let us consider $ \rho_{n,i} =\rho_{n} $ and $\gamma_{n,i} = \gamma_{n}$ where 
\begin{align*}
\rho_{n} = \left(1-\frac{\Poiss_1[\R]}{n}\right)\delta_{0}+\frac{1}{n}\Poiss_1	
\qquad \text{and} \qquad 
\gamma_{n} = \left(1-\frac{\Poiss_2[\R]}{n}\right)\delta_{0}+\frac{1}{n}\Poiss_2	. 
\end{align*}
Thus, from Lemma \ref{lemma:convolution_wasserstein-inequality} and the fact that the free Wasserstein distance is a refinement of the classical one, we have that 
\begin{align*}
\freeWass{1}\left( \rho_{n,1} \circledast \cdots \circledast \rho_{n,n},  \gamma_{n,1} \circledast \cdots \circledast \gamma_{n,n} \right)
\leq 
n \cdot \freeWass{1}\left( \rho_{n} , \gamma_{n} \right)
\leq 
n \cdot \classWass{1}\left( \rho_{n} , \gamma_{n} \right). 
\end{align*}
Since $\Poiss_1[\R]=\Poiss_{2}[\R]$, any tensor transport plan $\vartheta$ from $\mathbb{\Pi}(\Poiss_1,\Poiss_2)$ gives rise to a tensor transport $\tilde{\vartheta}$ belonging to $\mathbb{\Pi}[\rho_{n},\gamma_{n}]$ given by 
\begin{align*}
\tilde{\vartheta}
= 
\left(1-\frac{\Poiss_1[\R]}{n}\right)\delta_{(0,0)}
+\frac{1}{n}\vartheta,
\end{align*}
which satisfies 
\begin{align*}
 \classWass{1}\left( \rho_{n} , \gamma_{n} \right)
  &\leq  \int_{\R^{2}}|x-y|	\tilde{\vartheta}(dx,dy) 
  = \frac{1}{n} \int_{\R^{2}}|x-y|{\vartheta}(dx,dy) .
\end{align*}
In particular, by taking $\vartheta \in \mathbb{\Pi}(\Poiss_1,\Poiss_2)$ in such a way that 
$ %\begin{align*}
\int_{\R^{2}}|x-y|	\vartheta(dx,dy) = \classWass{1} ( \Poiss_1 , \Poiss_2 ) 
%\end{align*} 
$, we obtain 
\begin{align*}
\freeWass{1}\left( \rho_{n,1} \circledast \cdots \circledast \rho_{n,n},  \gamma_{n,1} \circledast \cdots \circledast \gamma_{n,n} \right)
\leq 
\classWass{1} ( \Poiss_1 , \Poiss_2 ) . 
\end{align*}
Since $\rho_n^{\circledast n}$ and $\gamma_n^{\circledast n}$ converge weakly to $\tau_{\Poiss_1}$ and $\tau_{\Poiss_2}$, respectively, the lower semicontinuity of $\freeWass{1}$ together with the preceding estimate yields
\begin{align}\label{eq:prevtogetapch}
\freeWass{1}(\tau_{\Poiss_1},\tau_{\Poiss_2})
\leq
\classWass{1}(\Poiss_1,\Poiss_2).
\end{align}
Now observe that the centered laws $\pi_{\Poiss_1}$ and $\pi_{\Poiss_2}$ are obtained from $\tau_{\Poiss_1}$ and $\tau_{\Poiss_2}$ by translations of sizes $m_1[\Poiss_1]$ and $m_1[\Poiss_2]$, respectively. The triangle inequality and the translation invariance of $\freeWass{1}$ then yields
\begin{align*}
\freeWass{1}(\pi_{\Poiss_1},\pi_{\Poiss_2})
&\leq
\freeWass{1}(\tau_{\Poiss_1},\tau_{\Poiss_2})
+
|m_1[\Poiss_1]-m_1[\Poiss_2]|.
\end{align*}
Combining this estimate with \eqref{eq:prevtogetapch}, we obtain \eqref{eqn:convolution_inequality}.
\end{proof}

%--------------------------------------------------------------------------------------%
%--------------------------------------------------------------------------------------%
%                             SECTION:MAIN_THEOREM                                     %
%--------------------------------------------------------------------------------------%
%--------------------------------------------------------------------------------------%

% \newpage 
 
\section{Proof of Theorem \ref{thm:main_theorm}}\label{section:main_theorem}

The proof of our main result is divided in seven steps, 
comparing $\mu_0 = g_{\sigma^2}  \circledast \pi_{\Poiss} \circledast \delta_a $ and $\mu_7:= g_{\zeta^2}  \circledast\pi_{\Delta} \circledast \delta_b$ through six intermediate approximations $\mu_1,\mu_2,\ldots,\mu_6$ and a repeated use of the triangle inequality. 
For arbitrary $\mu_1,\mu_2,\ldots,\mu_6$, it holds that 
\[
\freeWass{1}(\mu_0,\mu_7)
\leq 
\freeWass{1}(\mu_0,\mu_1) + \freeWass{1}(\mu_1,\mu_2) + \cdots + \freeWass{1}(\mu_6,\mu_7). 
\]
The contribution of the $k$-th step relies on finding an appropriate choice of $\mu_k$ and determining an upper bound for $$\freeWass{1}(\mu_{k-1},\mu_{k}).$$

Throughout the proof of Theorem \ref{thm:main_theorm}, the convolution $\circledast\in\{\ast,\boxplus,\uplus\}$ is fixed. For notational simplicity, we suppress the superscript $\circledast$ in the cumulants and write $\kappa_n[\rho]$ instead of $\kappa_n^{\circledast}[\rho]$. 
\begin{Notation}
From now on, for $\mu = g_{\sigma^2}  \circledast \pi_{\Poiss} \circledast \delta_a$ and $\nu = g_{\zeta^2}  \circledast\pi_{\Delta} \circledast \delta_b$  as in Theorem \ref{thm:main_theorm}, we take 
\begin{align}
\mathcal{E}_{\nu}[\mu] : = \sum_{n=1}^{2m+4}  \abs{ \kappa_n[\mu] - \kappa_n[\nu]}  . 
\end{align}
Additionally, we consider  
\begin{equation}\label{def:U_Delta_L_Delta}
U_\Delta =  \max \{ 1, \abs{x_1} , \ldots, \abs{x_m} \}, 
\end{equation}
and setting $x_0 =0$, we let 
\begin{equation}\label{def:D_Delta_d_Delta}
D_\Delta =  \max \{ 1, \abs{x_i-x_j} : 0 \leq i < j \leq m \}  \text{\quad and \quad}  3 r_\Delta = d_\Delta =  \min \{1, \abs{x_i -x_j} : 0 \leq i < j \leq m \}  .  
\end{equation}
It is immediate that $d_\Delta \leq L_\Delta \leq U_\Delta \leq D_\Delta \leq 2 U_\Delta$.
\end{Notation}

\begin{Remark}
By an approximation argument, it suffices to show the result in the case where $\Poiss\in\mathcal{M}^{\infty}(\R^{*})$, the set of finite measures over $\R^* = \R \setminus \{0\}$ with finite moments of all orders. 
Moreover,  without loss of generality, we can assume that $\Poiss[I]>0$ for every non-empty open interval $I\subset \R$. 
Indeed, if the latter condition does not hold, we simply replace $\mu=g_{\sigma^2}  \circledast \pi_{\Poiss} \circledast \delta_a$ with the infinitely divisible distribution $\mu_{\varepsilon} = g_{\sigma^2}  \circledast \pi_{\Poiss_{\varepsilon}} \circledast \delta_a$ where $\Poiss_{\varepsilon}=\Poiss(dx)+\varepsilon \exp\{-x^{2}\}dx$. 
Then, if \eqref{eqn:main_cumulant_inequality} holds for $\mu_{\varepsilon}$, and we let $\varepsilon\rightarrow0$, we obtain that it also holds for $\mu$ by continuity. 
%
%Consequently,  without loss of generality, we can assume that $\Poiss[I]>0$ for every non-empty open interval $I$. 
\end{Remark}

Before proceeding with the first step, we need the following lemma that establises a relation between the cumulant expression in \eqref{eqn:main_cumulant_inequality} from  Theorem \ref{thm:main_theorm} and the integral of a certain non-negative polynomial with respect to \( \eta \).
%

%\textcolor{purple}{\[U_\Delta =  \max \{ 1, \abs{x_1} , \ldots, \abs{x_m} \}  \text{\quad and \quad}L_\Delta =  \min \{1, \abs{x_1} , \ldots, \abs{x_m} \} \]} \textcolor{purple}{For the following, we take $x_0 =0$.} \textcolor{purple}{\[D_\Delta =  \max \{ 1, \abs{x_i-x_j} : 0 \leq i < j \leq m \}  \text{\quad and \quad}d_\Delta =  \min \{1, \abs{x_i -x_j} : 0 \leq i < j \leq m \} \] Note that \[d_\Delta \leq L_\Delta \leq 1 \leq U_\Delta \leq D_\Delta \]}

\begin{Lemma}\label{lemma:polynomial_cumulant_bound}
Let $\mu = g_{\sigma^2}  \circledast \pi_{\Poiss} \circledast \delta_a$ and $\nu = g_{\zeta^2}  \circledast \pi_{\Delta} \circledast \delta_b$ be infinitely divisible distributions as in Theorem \ref{thm:main_theorm} with triples $(a,\sigma^2,\Poiss)$ and $(b,\zeta^2,\Delta)$, respectively. 
Then
\begin{align}\label{eqn:polynomial_cumulant_inequality}
(x_1 \cdots x_m)^2 (\sigma^2 - \zeta^2)
+
\int_{\R} x^2 (x-x_1)^2  \cdots (x-x_m)^2 \Poiss (dx)
\leq  
	(2 U_{\Delta})^{2m} \, \sum_{n=2}^{2m+2} \abs{ \kappa_n[\mu] - \kappa_n[\nu] },
\end{align}
where $U_\Delta =  \max \{ 1, \abs{x_1} , \ldots, \abs{x_m} \}$.
\end{Lemma}

\begin{proof}
Take $q(x) = x^2 (x-x_1)^2 \cdots (x-x_m)^2$ and write 
\[
%x^2 (x-x_1)^2  \cdots (x-x_m)^2 
q(x) = (x_1 \cdots x_m)^2 x^2 + \sum_{n=3}^{2m+2} \alpha_{n} x^{n}
\text{\quad with \quad} 
{ \textstyle 
\alpha_{n} \, =\sum_{ \mathbb{i} } \left(  \prod_{\ell=1}^{m}(-x_{\ell})^{i_{2\ell-1}+i_{2\ell}} \right) }
\]
where the latter sum runs over all the multi-indices $\mathbb{i}=(i_1,i_2,\dots, i_{2m}) \in \{0,1\}^{2m}$ satisfying \[ i_1+ i_2 + \cdots+i_{2m}=2m+2-n .\]
Since the measure $\Delta$ is supported on the set $\{x_1, x_2, \ldots, x_m \}$, and $q(x_k)=0$ for $1\leq k\leq m$, we obtain 
\begin{align*}
\int_{\R}  q (x) \, \Poiss (dx) 
  &= \int_{\R}  q (x) \, (\Poiss-\Delta)(dx)
%\\ &
=(x_1 \cdots x_m)^2 \left(m_2[\Poiss] - m_2[\Delta]  \right)  +  \sum_{n=3}^{2m+2}  \alpha_n \cdot \left(m_n[\Poiss] - m_n[\Delta]  \right) .
  % \int_{\R}  p(x) (\Poiss-\Delta) (dx) .
\end{align*}
Now, from \eqref{eq:cumulantsidd}, we know that $m_2[\Poiss] =\kappa_2(\mu)-\sigma^2$,  $m_2[\Delta] =\kappa_2(\nu)-\zeta^2$, and $\kappa_n[\mu] = m_n[\Poiss]$  and $\kappa_n[\nu] = m_n[\Delta]$ for $n \geq 3$. 
Thus, replacing moments by cumulants, we obtain 
\[
\int_{\R} q(x) \, \Poiss (dx)  \leq  
(x_1 \cdots x_m)^2 (\kappa_2[\mu]-\sigma^{2} - \kappa_2[\nu]+\zeta^2) + 
\sum_{n=3}^{2m+2} \abs{\alpha_n} \abs{ \kappa_n[\mu] - \kappa_n[\nu] } , 
\]
and consequently 
\begin{align}\label{eqn:polynomial_cumulant_inequalityprev}
\int_{\R} x^2 (x-x_1)^2  \cdots (x-x_m)^2 \Poiss (dx) %\leq C \mathcal{E}_{\Delta} (\mu) = 
\leq C \sum_{n=2}^{2m+2} \abs{ \kappa_n[\mu] - \kappa_n[\nu] }+(x_1 \cdots x_m)^2 \left( \zeta^2 -\sigma^{2}   \right)
\end{align}
with $C = \max\{ (x_1 \cdots x_m)^2, \abs{\alpha_3}, \ldots , \abs{\alpha_{2m+2}} \}$. 
Finally, we know that 
\[ 
C \leq (2U_\Delta)^{2m}. 
%\text{\quad where \quad} U_\Delta =  \max \{ 1, \abs{x_1} , \ldots, \abs{x_m} \}
\]
since we have the inequalities 
\[
(x_1\cdots x_m)^{2}\leq U_\Delta^{2m} 
\text{\quad and \quad} {\textstyle
\abs{\prod_{\ell=1}^{m}(-x_{\ell})^{i_{2\ell-1}+i_{2\ell}} } 
\leq U_\Delta^{2m}}
\text{\quad for any \quad} (i_1,\dots, i_{2m}) \in \{0,1\}^{2m}.
\]
%for any $(i_1,\dots, i_{2m}) \in \{0,1\}^{2m}$
%
%
%
\end{proof}

%removecomment \fcomm{5. Review statement of Theorem \ref{thm:main_theorm} and Definition \ref{def:compound_poisson} to take into account the following paragraph.}

%\end{comment}

%\newpage 
%-----------------------------------------------------%
%-------------------- Step I -------------------------%
%-----------------------------------------------------%

%\jcomment{A ver si podemos juntarnos y revisar esta sección con cuidado. Me acabo de dar cuenta que cambié la notación original, más las distintas modificaciones que sufrió el resutlado orignal. Por ejemplo, creo que antes era la distancia entre $\mu$ y $\Poiss$ y ahora está como $\mu_1$ y $\mu_2$ en la introducción. Igual voy checando lo que pueda yo sólo con cuidado} 

\noindent \textbf{Step I.} %\\ 
Define  $\mu_0 := g_{\sigma^2} \circledast \pi_{\Poiss} \circledast \delta_a$ and $\mu_1 := g_{\zeta^2} \circledast\pi_{\Poiss} \circledast \delta_a$.  Due to Lemma \ref{lemma:convolution_wasserstein-inequality}, we obtain  
\begin{align*}
  \freeWass{1} ( \mu_0, \mu_1 ) 
  \leq  \freeWass{1}( g_{\sigma^2} , g_{\zeta^2}). 
\end{align*}
For a bound of the right-hand side, since $\sigma^2\geq\zeta^2$, we first consider two $\circledast$-independent centered $\circledast$-Gaussian variables $X=X^*$ and $Y=Y^*$ with variances $\zeta^2$ and $\sigma^2-\zeta^2$, respectively, defined in a $C^*$-probability space $(\mathcal{A},\varphi)$. 
By \eqref{eqn:linearization_cumulants} and \eqref{eqn:gaussian_in_cumulants}, the variable $X+Y$ is centered and $\circledast$-Gaussian with variance $\sigma^2$, so $(X,X+Y)$ is a non-commutative coupling for the pair of measures $( g_{\zeta^2},g_{\sigma^2})$. 
Hence, from the definition of $\freeWass{1}$, it follows that 
\begin{align*}
\freeWass{1}( g_{\zeta^2},g_{\sigma^2})
  &\leq 
  \varphi(\abs{Y})
  %\leq \varphi[Y^2]^{1/2}=\sqrt{\sigma^2-\zeta^2}.
\end{align*}
where $\abs{Y}$ denotes the positive element in the $C^*$-algebra generated by $Y\in \mathcal{A}$ such that $\abs{Y}^2 = YY^* = Y^2$. 
The Cauchy–Schwarz inequality in $(\mathcal{A},\varphi)$ then gives 
\begin{align*}
    \varphi(\abs{Y} \cdot 1_{\mathcal{A}})^2 \leq \varphi(\abs{Y}^2) \cdot \varphi(1_{\mathcal{A}}) = \sigma^2 - \zeta^2, 
\quad \text{and consequently} \quad 
\freeWass{1}( g_{\zeta^2},g_{\sigma^2}) \leq \sqrt{\sigma^2-\zeta^2}.
\end{align*}
Let us now get a bound for $\sqrt{\sigma^2-\zeta^2}$. 
For this, take $q(x) = x^2 (x-x_1)^2 \cdots (x-x_m)^2$ and $p(x) = q(x) - (x_1 \cdots x_m)^2 x^2$. 
%\[ x^2 (x-x_1)^2  \cdots (x-x_m)^2 q(x) - (x_1 \cdots x_m)^2 x^2 + p(x). \]
%
Then, by Lemma \ref{lemma:polynomial_cumulant_bound}, we have that 
\begin{align*}
(x_1 \cdots x_m )^2 (\sigma^{2}-\zeta^{2})
\leq 
        (x_1 \cdots x_m )^2 (\sigma^{2}-\zeta^{2})+ \int_{\R}q(x) \Poiss(dx)
\leq 
             (2U_\Delta)^{2m} \cdot 
            \sum_{n=2}^{2m+2} \abs{ \kappa_n[\mu] - \kappa_n[\nu] }.
\end{align*}
% \begin{align*}
% (x_1 \cdots x_m )^2 (\sigma^{2}-\zeta^{2})
%   &\leq (x_1 \cdots x_m )^2 (\sigma^{2}-\zeta^{2})+(x_1 \cdots x_m )^2 \int_{\R}q(x) \Poiss(dx)\\
%   &\leq (x_1 \cdots x_m)^2(\kappa_2[\mu]-\zeta^{2}-\kappa_2[\tau_{\Delta}])
%   +4^{m}(1+|\bm{x}|)^{2m} \sum_{k=3}^{2m+2} \abs{ \kappa_n(\mu) - \kappa_n(\tau_{\Delta}) }.
% \end{align*}
% Since $\kappa_2[\upsilon]=\zeta^{2}+\kappa_2[\tau_{\Delta}]$, we deduce the inequality
% \begin{align*}
% (x_1 \cdots x_m )^2 (\sigma^{2}-\zeta^{2})
%   &\leq (x_1 \cdots x_m)^2(\kappa_2[\mu]-\kappa_2[\upsilon])
%   +4^{m}(1+|\bm{x}|)^{2m} \sum_{k=3}^{2m+2} \abs{ \kappa_n(\mu) - \kappa_n(\tau_{\Delta}) }.
% \end{align*}
% On the other hand, 
% \begin{align*}
% (x_1 \cdots x_m )^2 (\zeta^2 -\sigma^{2})
%   &\leq 4^{m}(1+|\bm{x}|)^{2m} \sum_{k=2}^{2m+2} \abs{ \kappa_n(\mu) - \kappa_n(\tau_{\Delta}) }
% \end{align*}
Since $r_\Delta  = d_\Delta/3 \leq \abs{x_j}$ for any $j$, we have $0 < r^{2m}_\Delta \leq (x_1 \cdots x_m )^2$, and hence 
\begin{align} \label{eqn:wasserstein_triangle_1}
  \freeWass{1} ( \mu_0, \mu_1 )
\leq \sqrt{\sigma^2-\zeta^2}  \leq 
%  \leq \abs{ \sigma}
\left(  \frac{2 \, U^{}_\Delta}{ r_\Delta } \right)^{m} {\color{black}\mathcal{E}_{\nu}[\mu]^{1/2}}
  %\leq   2^{m}(1+R_{\Delta})^{m} \mathcal{E}_{\upsilon}[\mu]^{1/2}.
\end{align}
%

%\textcolor{purple}{\[   \freeWass{1} ( \mu_0, \mu_1 )  \leq \frac{2^m U^{m}_\Delta}{L^{m}_\Delta} \mathcal{E}_{\upsilon}[\mu]^{1/2} = \left(  \frac{ 2 \, U^{}_\Delta}{L^{}_\Delta} \right)^{m} \mathcal{E}_{\upsilon}[\mu]^{1/2}  \leq \left(  \frac{ 2 \, U^{}_\Delta}{ \radio_{\Delta} } \right)^{m} \mathcal{E}_{\upsilon}[\mu]^{1/2}  \] }
%$0\leq\zeta^2\leq \sigma^2$
%\newpage 

%-----------------------------------------------------%
%-------------------- Step II ------------------------%
%-----------------------------------------------------%
\noindent \textbf{Step II.} %\\
Define $\mu_1:= g_{\zeta^{2}} \circledast\pi_{\Poiss} \circledast \delta_a  \circledast \delta_0$ and take $\mu_2:= g_{\zeta^{2}} \circledast\pi_{\Poiss} \circledast \delta_a  \circledast \delta_{b-a}$. 
Due to Lemma \ref{lemma:convolution_wasserstein-inequality} and \eqref{Wasserstein:frevsclass}, we obtain 
\begin{align*}
  \freeWass{1} ( \mu_1, \mu_2 ) 
  \leq  \freeWass{1}( \delta_0 , \delta_{b-a} ) 
  \leq  \classWass{1}( \delta_0 , \delta_{b-a} ) .
\end{align*}
But, we know that $ \classWass{1}( \delta_0 , \delta_{b-a} )\leq \abs{b-a} =  \abs{\kappa_1[\mu]-{\color{black} \kappa_1[\nu]} }$. 
And consequently, 
\begin{align}\label{eqn:wasserstein_triangle_2}
  \freeWass{1} ( \mu_1, \mu_2 ) 
{\color{black} \leq \mathcal{E}_{\nu}[\mu]
\leq \mathcal{E}_{\nu}[\mu]^{1/2}}
\end{align}
where the last inequality follows from the condition 
${\color{black} 0 \leq \mathcal{E}_{\nu}[\mu]\leq 1}$.\\

%-----------------------------------------------------%
%-------------------- Step III -----------------------%
%-----------------------------------------------------%
\noindent \textbf{Step III.} %\\
%
%From this step onward, we choose any $\radio_{\Delta}> 0$ such that  \begin{align*}   \radio  \leq \min_{1 \leq k \leq m } \frac{ \abs{x_k} } {2} . \end{align*} \textcolor{cyan}{We can take any $\radio < \tfrac{d_\Delta}{2}$ }
Take $I = [- \radio_{\Delta} , \radio_{\Delta} ]$. 
Consider the measures $\Poiss_{3}$ and $\tilde{\Poiss}_{3}$ determined by  
\begin{align*}
    \Poiss_{3}(dx):=\Indi{\R \backslash I}(x)\Poiss(dx)
    \quad \text{and} \quad
    \tilde{\Poiss}_{3}(dx):=\Indi{I}(x)\Poiss(dx), 
\end{align*}
respectively. 
Write $\mu_2 = g_{\zeta^{2}}\circledast\pi_{\tilde{\Poiss}_3} \circledast \pi_{\Poiss_3} \circledast \delta_{b} $ and let $\mu_3 = g_{\zeta^{2}}\circledast\delta_{0} \circledast \pi_{\Poiss_3} \circledast \delta_{b}  $. 
Due to Lemma \ref{lemma:convolution_wasserstein-inequality}, we know that 
\begin{align*}
  \freeWass{1} ( \mu_2, \mu_3 ) 
  \leq  \freeWass{1}( \pi_{\tilde{\Poiss}_3} , \delta_0 ) 
  \leq  \classWass{1}( \pi_{\tilde{\Poiss}_3} , \delta_0 ) .
\end{align*}
Since $m_2[ \pi_{\tilde{\Poiss}_3} ] =  \kappa_2[ \pi_{\tilde{\Poiss}_3}] = m_2 [ \tilde{\Poiss}_3 ] $, the Cauchy-Schwartz inequality yields  
\begin{align*}
     \classWass{1}( \pi_{\tilde{\Poiss}_3} , \delta_0 ) 
\leq
    \int_{\R} |x|  \pi_{\tilde{\Poiss}_3} (dx)
\leq     
\pi_{\tilde{\Poiss}_3} [\R]^{1/2} 
    \left(\int_{\R} x^{2} \  \pi_{\tilde{\Poiss}_3} (dx)\right)^{1/2}
=
    \left(\int_{\R} x^{2}   \  \tilde{\Poiss}_3 (dx)\right)^{1/2}.
\end{align*}
Moreover, since $1  \leq \abs{x-x_k} / \radio_{\Delta} $ for any $x \in I$ and $k=1,\ldots,m$, we have that 
\begin{align*}
    \int_{\R} x^{2}  \ \tilde{\Poiss}_3 (dx)
\leq 
   \radio_{\Delta}^{-2m}  \int_{I}  x^{2} (x-x_1)^2 \cdots (x-x_m)^2  \ \Poiss (dx)
\leq 
   \radio_{\Delta}^{-2m}  \int_{\R}  x^{2} (x-x_1)^2 \cdots (x-x_m)^2  \ \Poiss (dx). 
\end{align*}
It then follows from Lemma \ref{lemma:polynomial_cumulant_bound} that 
\begin{align} \label{eqn:wasserstein_triangle_3}
\freeWass{1} ( \mu_2 , \mu_3 )  \leq \left(  \frac{ 2 \, U^{}_\Delta}{ \radio_{\Delta} } \right)^{m}  \mathcal{E}_{\upsilon}[\mu]^{1/2}
\end{align}

%\textcolor{purple}{\[   \freeWass{1} ( \mu_2 , \mu_3 )  \leq \frac{2^m(1+R_{\Delta})^{m}}{ r_{\Delta}^{m} } \mathcal{E}_{\upsilon}[\mu]^{1/2}.  \]}

%\newpage 

%-----------------------------------------------------%
%-------------------- Step IV ------------------------%
%-----------------------------------------------------%

\noindent \textbf{Step IV.} %\\
%
%Since $\Poiss\in\mathcal{M}^{\infty}(\R^{*})$, we have that $\Poiss[ \R \backslash \{0\}] > 0$. 
%
Take $J_k = [x_k-\radio_{\Delta},x_k+\radio_{\Delta}]$ for $k = 1,\ldots,m$ and let  $\Poiss_{4}$ denote the measure given by 
\begin{align*}
	\Poiss_4(dx)=\Indi{J}(x) \, \Poiss(dx)  
 \quad \text{where} \quad 
    J= \bigcup_{k=1}^m J_k.
\end{align*}
Notice that the intervals $J_k$ are mutually disjoint and $J$ is contained in $\R\backslash I$ with $I = [- \radio_{\Delta} , \radio_{\Delta} ]$ since $\radio_{\Delta} < d_{\Delta}/2$. 
Thus, we have  
\[ 
(\Poiss_3 - \Poiss_4) (dx) = \Indi{\R \backslash (I \cup J)}(x) \, \Poiss(dx) .
%\quad \text{and} \quad 1 \leq \abs{x_k - x}/ \radio_{\Delta}  \quad \text{for} \quad x \in \R \backslash (I \cup J)  \quad \text{and} \quad k=1,\ldots,m .
\]
%
%for $x \in \R \backslash (I \cup J) $ and $k=1,\ldots,m$.
%
Moreover, $\Poiss$ charges strictly positive mass to open intervals, %so we have that $\Poiss_4[\R] = \Poiss[J] > 0$ and $(\eta_3 - \eta_4)[\R] =  \eta[R\setminus(I\cup J)] > 0$. 
and consequently $\Poiss_4, \Poiss_3 - \Poiss_4\in\Mca(\R)$. 
Hence, we can write $\mu_3 = g_{\zeta^{2}} \circledast\pi_{\Poiss_3 - \Poiss_4} \circledast \pi_{\Poiss_4} \circledast \delta_{b} $ and take $\mu_4 = g_{\zeta^{2}} \circledast\delta_0  \circledast \pi_{\Poiss_4} \circledast  \delta_{b}$. 
Due to Lemma \ref{lemma:convolution_wasserstein-inequality}, we have that 
\begin{align*}
  \freeWass{1} ( \mu_3, \mu_4 ) 
  \leq  \freeWass{1}( \pi_{\Poiss_3 - \Poiss_4} , \delta_0 ) 
  \leq  \classWass{1}  ( \pi_{\Poiss_3 - \Poiss_4} , \delta_0 )   .
\end{align*}
Now, since $m_2[ \pi_{\Poiss_3 - \Poiss_4} ] =  \kappa_2[ \pi_{\Poiss_3 - \Poiss_4} ] = m_2 [ \Poiss_3 - \Poiss_4 ] $, the Cauchy-Schwartz inequality yields  
\begin{align*}
     \classWass{1}( \pi_{\Poiss_3 - \Poiss_4}, \delta_0 ) 
& \leq
    \int_{\R} |x| \ \pi_{\Poiss_3 - \Poiss_4} (dx)
\leq 
\pi_{\Poiss_3 - \Poiss_4}  [\R]^{1/2} 
    \left(\int_{\R} x^{2} \  \pi_{\Poiss_3 - \Poiss_4} (dx)\right)^{1/2} 
%\\& \leq 
=
    \left(\int_{\R} x^{2}   \  (\Poiss_3 - \Poiss_4) (dx)\right)^{1/2}.
\end{align*}
But, we have $1  \leq \abs{x-x_k} / \radio_{\Delta} $ for any $x \in \R \backslash (I \cup J) $ and $k=1,\ldots,m$, so we obtain  
\begin{align*}
    \int_{\R} x^{2}   \  (\Poiss_3 - \Poiss_4) (dx)
\leq 
    \radio^{-2m}_{\Delta}  \int_{ \R \backslash (I \cup J) }  x^{2} (x-x_1)^2 \cdots (x-x_m)^2  \  \Poiss (dx). 
\end{align*}
As in \eqref{eqn:wasserstein_triangle_3}, it follows from Lemma \ref{lemma:polynomial_cumulant_bound} that 
\begin{align} \label{eqn:wasserstein_triangle_4}
  \freeWass{1} ( \mu_3 , \mu_4 )  \leq \left(  \frac{ 2 \, U^{}_\Delta}{ \radio_{\Delta} } \right)^{m}  {\color{black} \mathcal{E}_{\nu}[\mu]^{1/2}} 
\end{align}
%

%\textcolor{purple}{\[   \freeWass{1} ( \mu_3 , \mu_4 )   \leq \frac{2^m(1+R_{\Delta})^{m}}{ r_{\Delta}^{m} } \mathcal{E}_{\upsilon}[\mu]^{1/2}.  \]}

%\newpage 
%-----------------------------------------------------%
%-------------------- Step V -------------------------%
%-----------------------------------------------------%

\noindent \textbf{Step V.} %\\
As in the previous step, take $J_k = [x_k-\radio_{\Delta},x_k+\radio_{\Delta}]$ for $k = 1,\ldots,m$, 
and consider now the discrete measure $\Poiss_5$ given by 
\begin{align*}
\Poiss_5
  & =\sum_{k=1}^{m}\Poiss_4[J_k]\delta_{x_k}. 
\end{align*}
Since the intervals $J_k$ are disjoint, it holds that $\Poiss_4[\R] = \Poiss_5[\R]$.
Write $\mu_4 =  g_{\zeta^{2}} \circledast\pi_{\Poiss_4} \circledast \delta_{b} $ and take $\mu_5 = g_{\zeta^{2}} \circledast\pi_{\Poiss_5} \circledast \delta_{b}$. 
Due to Lemmas \ref{lemma:convolution_wasserstein-inequality} and \ref{lemma:convolution_inequality}, we obtain that
\begin{align*}
  \freeWass{1} ( \mu_4, \mu_5 ) 
  \leq  \freeWass{1}( \pi_{\Poiss_4}  , \pi_{\Poiss_5}  ) 
  \leq  \classWass{1}  ( \Poiss_4 , \Poiss_5 )
  {\color{black} +
  |m_1[\Poiss_4]-m_1[\Poiss_5]|}.
\end{align*}
To obtain a bound for the first term in the rightmost expression, we consider a particular tensor transport plan between $\Poiss_4$ and $\Poiss_5$ constructed as follows.  
Since the intervals $J_k$ are disjoint, the mapping $T:\R\rightarrow\R$ given by 
\begin{align*}
T(x) = \left\{ \begin{array}{cc}
    x_k, & \text{if } x \in J_k,  \\
     0, & \text{otherwise,} 
\end{array}
\right.
\end{align*}
is well-defined. 
Note that $\Poiss_5$ coincides with the push-forward of $\Poiss_4$ under the mapping $T$. %, $B \mapsto \Poiss_4 (T^{-1}(B))$. 
Consequently, the push-forward of $\Poiss_4$ under the mapping  $x\mapsto(x,T(x))$ is a tensor transport plan between $\Poiss_4$ and $\Poiss_5$, yielding 
\begin{align*}
\classWass{1}  ( \Poiss_4 , \Poiss_5 ) 
   &\leq \int_{\R}|x-T(x)|  \ \Poiss_4(dx).	
 \end{align*}
{\color{black} Moreover, since $T_{\#}\Poiss_4=\Poiss_5$, we have that
\begin{align*}
|m_1[\Poiss_4]-m_1[\Poiss_5]|
=
\left|\int_{\R}(x-T(x))\,\Poiss_4(dx)\right|
\leq
\int_{\R}|x-T(x)|\,\Poiss_4(dx),
\end{align*}
{\color{black} yielding} the inequality

\begin{align*}
  \freeWass{1} ( \mu_4, \mu_5 ) 
  \leq
  2\int_{\R}|x-T(x)|\,\Poiss_4(dx).
\end{align*}}
Thus, since $\Poiss_4(dx)=\Indi{J}(x) \, \Poiss(dx)$, with $J$ the disjoint union of the $J_k$, the Cauchy-Schwartz inequality and the definition of  $T$  give
\begin{align*}
{\color{black}    \freeWass{1}(\mu_4,\mu_5)}
\leq 
    {\color{black} 2}\Poiss_4[\R]^{1/2}
    \left(\int_{\R}|x-T(x)|^{2} \ \Poiss_4(dx)\right)^{1/2} 
= 
    {\color{black} 2}\Poiss_4[\R]^{1/2}
    \left({\color{black}\sum_{k=1}^{m}}\int_{J_k}|x-x_k|^{2} \ \Poiss_4(dx)\right)^{1/2}.	
\end{align*}
Notice that $x \in J_k$ implies $1 \leq \abs{x} / \radio_{\Delta} $ and $1 \leq \abs{x-x_{\ell}} / \radio_{\Delta} $ for any ${\ell} \neq k$. 
Hence, we get that 
\begin{equation}\label{eqn:step_v_integral_interval_Jk}
\int_{J_k}|x-x_k|^{2} \ \Poiss_4(dx) 
\leq 
\radio^{-2m}_{\Delta} \int_{J_k} x^2 (x-x_1)^2 \cdots (x-x_m)^2 \ \Poiss_4(dx) .
\end{equation}
Combining the last two inequalities, we get    
\begin{align*}
{\color{black}\freeWass{1}  ( \mu_4 , \mu_5 )}
\leq 
    \frac{{\color{black} 2}\Poiss_4[\R]^{1/2}}{\radio^m_{\Delta} }
    \left( \int_{J} x^2 (x-x_1)^2 \cdots (x-x_m)^2 \ \Poiss (dx) \right)^{1/2}.	
\end{align*}
As in \eqref{eqn:wasserstein_triangle_3} and \eqref{eqn:wasserstein_triangle_4}, it follows from Lemma \ref{lemma:polynomial_cumulant_bound} that 
\begin{align*} \label{eqn:wasserstein_triangle_5}
\freeWass{1} ( \mu_4 , \mu_5 )  & \leq {\color{black} 2}\Poiss_4[\R]^{1/2} \left(  \frac{ 2 \, U^{}_\Delta}{  \radio_{\Delta} } \right)^{m}  {\color{black}\mathcal{E}_{\nu}[\mu]^{1/2}} 
%& \leq 2 m_2[\eta]^{1/2} \left(  \frac{ 2 \, U^{}_\Delta}{  \radio_{\Delta} } \right)^{m} \mathcal{E}_{\Poiss}[\mu]^{1/2} 
\end{align*}
Now, since $\Poiss_4(dx)=\Indi{J}(x) \, \Poiss(dx)$, with $J$ the disjoint union of the $J_k$, and $1 \leq \abs{x} / \radio_{\Delta} $ for any $x \in J_k$ and  $k=1,2,\ldots,m$, we obtain 
\[
\eta_4 [\R] \leq  r^2_{\Delta} \int_{J} x^2 \,  \eta(dx) \leq r^2_{\Delta} \, m_2 [\eta]. 
\]
Therefore,  the last two inequalities and the fact that $r_\Delta \leq 1$ yield 
\begin{align} \label{eqn:wasserstein_triangle_5}
\freeWass{1} ( \mu_4 , \mu_5 )  
 \leq 2 m_2[\eta]^{1/2} \left(  \frac{ 2 \, U^{}_\Delta}{  \radio_{\Delta} } \right)^{m} \mathcal{E}_{\nu}[\mu]^{1/2} 
\end{align}
%

%\newpage 

%-----------------------------------------------------%
%-------------------- Step VI ------------------------%
%-----------------------------------------------------%

\noindent\textbf{Step VI.} %\\
Let $ c= \Delta[\R]  /  \Poiss_5[\R] > 0 $. 
Here we will compare $\mu_5 =  g_{\zeta^{2}}\circledast\pi_{ \Poiss_5 } \circledast \delta_{b} $ and $\mu_6 = g_{\zeta^{2}} \circledast \pi_{c \Poiss_5 }  \circledast \delta_{b}  $. 
If $c=1$, then  $\mu_5$ and $\mu_6$ coincide, and we can proceed to the following step. 
For the case $c>1$, we write the measures $\mu_5$ and $\mu_6$ as   
\begin{align*}
\mu_5
  =g_{\zeta^{2}}\circledast\pi_{ \Poiss_5 }\circledast\delta_{0} \circledast \delta_{b} 
\text{\quad and \quad}
\mu_6
  = g_{\zeta^{2}} \circledast \pi_{ \Poiss_5 }\circledast \pi_{(c-1) \Poiss_5 }  \circledast \delta_{b}.
\end{align*}
Then, by Lemma \ref{lemma:convolution_wasserstein-inequality}, we obtain 
\begin{equation}\label{ineq:W1pistepV}
  \freeWass{1} ( \mu_5, \mu_6 ) 
  \leq  \freeWass{1}( \pi_{ (c-1) \Poiss_5 }  , \delta_0  ) 
  \leq  \classWass{1}  ( \pi_{ |c-1| \Poiss_5 }  , \delta_0 )  .
\end{equation}
The same inequality can be obtained in the case $c<1$ by writing 
\begin{align*}
\mu_5
  =g_{\zeta^{2}}\circledast\pi_{c\Poiss_5 }\circledast \pi_{(1-c)\Poiss_5 } \circledast \delta_{b}
\text{\quad and \quad}
\mu_6
  = g_{\zeta^{2}} \circledast \pi_{ c\Poiss_5 }  \circledast\delta_0\circledast \delta_{b} 
\end{align*}
instead. 
In any case, the Cauchy-Schwartz inequality yields    
\begin{equation}\label{cauchy_step_vi}
     \classWass{1}( \pi_{ |1-c| \Poiss_5 }  , \delta_0 ) 
\leq
    \int_{\R} |x|  \ \pi_{ |1-c| \Poiss_5 } (dx)
\leq 
    \left(\int_{\R} x^{2} \  \pi_{ |1-c| \Poiss_5 }  (dx)\right)^{1/2} .
\end{equation}
Now, since $m_2[ \pi_{ |1-c| \Poiss_5 }] =  \kappa_2[  \pi_{ |1-c| \Poiss_5 } ]$, due to $\kappa_1[  \pi_{ |1-c| \Poiss_5 } ] = 0$, and $\kappa_2[  \pi_{ |1-c| \Poiss_5 } ] = m_2 [ |1-c| \Poiss_5] $, it follows from \eqref{ineq:W1pistepV} and \eqref{cauchy_step_vi} that 
\begin{align*}
 \freeWass{1} ( \mu_5, \mu_6 )   
\leq 
  \left(\int_{\R} x^{2} \  \left( |1-c| \Poiss_5 \right)  (dx)\right)^{1/2} 
= 
  \abs{1-c}^{1/2}
  m_2[\Poiss_5 ]^{1/2}. 
\end{align*}
But, by writing $\abs{1-c} = \frac{1}{ \Poiss_5[\R] } \abs{ \Poiss_5[\R] - \Delta[\R] }$, the previous inequality becomes 
\begin{align}\label{ineq:keyinstep6}
 \freeWass{1} ( \mu_5, \mu_6 )   
\leq 
  {\color{black} \abs{\Poiss_5[\R] - \Delta[\R] }^{1/2}} \cdot
  \frac{m_2[\Poiss_5 ]^{1/2}}{\Poiss_5[\R]^{1/2}}. 
\end{align}
% upper bound for each $|\Poiss_5[J_k]-\Delta[J_k]|$.
%
Therefore, to obtain an upper bound for $ \freeWass{1} ( \mu_5, \mu_6 )$, it suffices to find upper bounds for each $|\Poiss_5[J_k]-\Delta[J_k]|$ since we have  
\begin{equation}\label{eqn:upper_bound_R_vs_Jk}
\frac{m_2[\Poiss_5]}{\Poiss_5[\R]}
= 
\sum_{k=1}^m x_k^2 \frac{\Poiss_5[J_k]}{\Poiss_5[\R]} 
\leq m U^2_{\Delta}
\text{\quad and \quad}
|\Poiss_5[\R] - \Delta[\R]|
	\leq \sum_{k=1}^{m}|\Poiss_5[J_k]-\Delta[J_k]|.
\end{equation}
To this end, having fixed $k\in [m]$, we take  
\begin{equation}
s_1:= 1 + \frac{\Indi{x_k > 0}}{x_k} + \sum_{j=1}^{k-1} \frac{1}{x_k-x_j},
\text{\quad}
s_2:= -1 + \frac{\Indi{x_k < 0}}{x_k} + \sum_{j=k+1}^{m} \frac{1}{x_k-x_j},
\text{\quad and \quad} 
s : = \frac{1}{x_k} + \sum_{\substack{j=1 \\ j\neq k}}^{m} \frac{1}{x_k-x_j}.
\end{equation}
The $x_j$'s are all non-zero and can be assumed to be ordered, i.e, $x_1 < x_2 < \cdots < x_m$. 
In this case, we get $s_1 - 1 \geq 0$ and $s_2+1 \leq 0$; 
the $\pm 1$ terms are simply to ensure that both $s_1$ and $s_2$ are non-zero and their sum equals $s$.  
The cases $s=0$ and $s \neq 0$ are similar, the main difference being that the former requires no extra points $x_{m+1}$ and $x_{m+2}$.  
So, we will focus on the case $s \neq 0$. 

Assume $s \neq 0$ and set $x_{m+1} = x_k + s_1^{-1}$ and $x_{m+2} = x_k + s^{-1}_2$. 
Notice that 
\[
\frac{k+\Indi{x_k > 0}}{D_\Delta}\leq \abs{s_1} \leq  \frac{k + \Indi{x_k > 0}}{d_\Delta}
\text{\quad and \quad}
\frac{m-k+1+\Indi{x_k < 0}}{D_\Delta}\leq \abs{s_2} \leq  \frac{m - k +1 + \Indi{x_k < 0}}{d_\Delta} ,  
\]
and hence 
\begin{equation}\label{eqn:bound_for_extra_points}
d_\Delta \leq \abs{s^{-1}_1}, \abs{s^{-1}_2} \leq D_\Delta 
\text{\quad and \quad}
\abs{x_{m+1}}, \abs{x_{m+2}} \leq U_\Delta+D_\Delta .
\end{equation}
Consider now $f_k(x)$ given for each $x \in \R $ by 
\begin{align*}
    f_k(x) : =  \beta^{-1}_k  \cdot x \cdot \prod_{\substack{j=1 \\ j\neq k}}^{m+2} (x-x_j)
\quad \text{where} \quad
\beta_k := x_k \cdot \prod_{\substack{j=1 \\ j\neq k}}^{m+2} (x_k-x_j).
\end{align*}
Note that  $f_{k}(x_j)^2=\delta_{x_k,x_j}$. 
Thus, since both $\Delta$ and $\eta_5$ are supported on $\{x_1,\dots, x_{m}\}$, we obtain 
\begin{align*}
\int_{\R}
f_k(x)^2 \,
\Delta(dx)
= \Delta[\{x_k\}] =\Delta[J_k]
\qquad \text{and} \qquad
\int_{\R}f_k(x)^2 \,
\eta_5(dx)
=\eta_5[\{x_k\}]
=\eta_5[J_k].
\end{align*}
Consequently, from the fact that $\eta[J_k]=\eta_5[J_k]$, we get 
\begin{align*}%\label{eqn:step6_bounds_mass}
|\Delta[J_k]-\Poiss_5[J_k]| 
& =  \abs{\Delta[J_k]   - \int_{\R}f_k(x)^2 \, \eta(dx)  + \int_{\R}f_k(x)^2 \, \eta(dx) -\Poiss[J_k] } \\ 
& \leq 
    \left|
            \int_{\R}
                f_k(x)^2
            (\Delta-\Poiss)(dx)\right|
+ 
           \left| \int_{\R}
                f_k(x)^2
                -\Indi{J_k}(x)
         \ \Poiss(dx) \right| .
\end{align*}
Splitting into two complementary parts the rightmost term in the last expression, we obtain 
\begin{equation}\label{eqn:step6_bounds_mass_middle}
|\Delta[J_k]-\Poiss_5[J_k]| 
\leq 
    \left|
            \int_{\R}
                f_k(x)^2
            (\Delta-\Poiss)(dx)\right|
+  
    \int_{\R \setminus J_k}
f_k(x)^2  \ \Poiss(dx)  
+
    \int_{J_k}\left| 
        f_k(x)^2 -1 \right|  \ \Poiss(dx) .
\end{equation}
We will find upper bounds for each of the three terms in the right-hand side of \eqref{eqn:step6_bounds_mass_middle}.  
For the first term, let us note that   
\begin{align*}
\beta^{2}_k \cdot f_k(x)^2 =  x^2 \prod_{\substack{j=1 \\ j\neq k}}^{m+2}  (x-x_j)^2 
= (x_1 \cdots x_{k-1} x_{k+1} \cdots x_{m+2})^2  x^2  
+
\sum_{n=0}^{2m+1} \tilde{\alpha}_{2m+4-n}  x^{2m+4-n} 
\end{align*}
where 
%
%\begin{align*} \textstyle
$
\tilde{\alpha}_{2m+4-n} =\sum_{ \tilde{\mathbb{i}}_k } \prod_{\ell \neq k }(-x_{\ell})^{i_{2\ell-1}+i_{2\ell}}$, 
%\end{align*}
%
and this latter sum runs over all the multi-indices $\tilde{\mathbb{i}}_k=(i_1,\ldots, i_{2k-2},  i_{2k+1}, \ldots, i_{2m+4}) \in \{0,1\}^{2m+2}$  satisfying $i_1+ \cdots + i_{2k-2} +i_{2k+1} + \cdots+i_{2m+4}=n$. 
Thus, %by writing  
%
%\[ f_k(x)^2  = \beta^{-2}_k (x_1 \cdots x_{k-1}  x_{k+1} \cdots x_{m+2})^2 x^2 + \beta^{-2}_k  \sum_{n=3}^{2m+4} \tilde{\alpha}_n x^n, \]
%
we obtain 
\[
\abs{ \int_{\R} f_k(x)^2  \, (\Poiss-\Delta)(dx)  }
\leq 
\beta^{-2}_k (x_1 \cdots x_{k-1}  x_{k+1} \cdots x_{ m+2})^2 \left| m_2[\Poiss] - m_2[\Delta] \right|
+
\beta^{-2}_k \sum_{n=3}^{2m+4} \abs{\tilde{\alpha}_n}\abs{ m_n[\Delta] - m_n[\Poiss] } .
\]
We know from \eqref{eq:cumulantsidd} that $m_2[\Poiss] =\kappa_2(\mu)-\sigma^2$,  $m_2[\Delta] =\kappa_2(\nu)-\zeta^2$, and $\kappa_n[\mu] = m_n[\Poiss]$  and $\kappa_n[\nu] = m_n[\Delta]$ for $n \geq 3$.
%
%Moreover, from Step I, we know that {\color{red} $\abs{\zeta^2-\sigma^2}\leq C_\Delta \mathcal{E}_{\nu}[\mu]$.}
%
Consequently, we get 
\begin{equation*}
\abs{ \int_{\R} f_k(x)^2  \, (\Poiss-\Delta)(dx)  }
\leq 
\tilde{C}_k \left( 
\abs{ \zeta^2 -\sigma^{2} }
+\
 \sum_{n=2}^{2m+4} \abs{ \kappa_n[\mu] - \kappa_n[\nu] }  \right)
\end{equation*}
% If $0\leq\zeta\leq \sigma$ 
where $\tilde{C}_k = \beta^{-2}_k \cdot \max\{(x_1 \cdots x_{k-1}  x_{k+1} \cdots x_{ m+2})^2 , \abs{\tilde{\alpha}_3}, \abs{\tilde{\alpha}_4}, \ldots ,  \abs{\tilde{\alpha}_{2m+4}} \}.$ 
It follows from \eqref{eqn:wasserstein_triangle_1} in Step I that  
%\[ \abs{\sigma^2-\zeta^2}  \leq \left(  \tfrac{2 \, U^{}_\Delta}{ r_\Delta } \right)^{2m} \mathcal{E}_{\nu}[\mu]\]
\begin{equation}\label{eqn:step_vi_bound_1_pre}
\abs{ \int_{\R} f_k(x)^2  \, (\Poiss-\Delta)(dx)  }
\leq 
\tilde{C}_k  \left[ \left( \tfrac{2 \, U^{}_\Delta}{ r_\Delta } \right)^{2m}   + 1 \right] \mathcal{E}_{\nu}[\mu].
\end{equation}
For the middle term in the right-hand side of \eqref{eqn:step6_bounds_mass_middle}, let us note that $1  \leq \abs{x-x_k} / \radio_\Delta $ for $x \in \R \setminus J_k $, so we get 
\[
f_k(x)^2 \leq r^{-2}_\Delta \, f_k(x)^2 (x-x_k)^2  
%= (\beta_k \radio_\Delta )^{-2}  x^{2} (x-x_1)^2 \cdots (x-x_k)^2 (x-x_{k+1})^2 \cdots (x-x_{m+2})^2  =  
\text{\quad for \quad} x \in \R \setminus J_k, 
\]
and hence  
\begin{align*}%\label{eqn:stepvi_secondterm_part_1_v2}
    \int_{\R \setminus J_k}
f_k(x)^2  \ \Poiss(dx)  
& \leq 
   (\beta_k \radio_\Delta )^{-2}  \int_{ \R  }  x^{2} 
            \prod_{j=1}^{m+2}(x-x_j)^2  \  \Poiss (dx) .
%\nonumber \\ & \leq (2)^{2m+4} \frac{(U_\Delta)^{2m} (U_\Delta + D_\Delta)^4}{r_\Delta^{2m+6}} \mathcal{E}_{\nu}[\mu]
\end{align*}
Now, the measure $\Delta$ is supported on the set $\{x_1, x_2, \ldots, x_m \}$, and $\prod_{j=1}^{m+2}(x_i-x_j)^2 =0$ for $1\leq i\leq m$, so we obtain 
\begin{align*}
%(\beta_k \radio_\Delta )^{-2}  
\int_{ \R  }  x^{2} 
            \prod_{j=1}^{m+2}(x-x_j)^2  \  \Poiss (dx) 
 = 
%(\beta_k \radio_\Delta )^{-2}  
\int_{ \R  }  x^{2} 
            \prod_{j=1}^{m+2}(x-x_j)^2  \  (\Poiss-\Delta) (dx) 
 \leq  (2)^{2m+4} (U_\Delta)^{2m} (U_\Delta + D_\Delta)^4  \, \mathcal{E}_{\nu}[\mu]
\end{align*}
where the last inequality follows from similar arguments as in \eqref{eqn:polynomial_cumulant_inequality} and \eqref{eqn:polynomial_cumulant_inequalityprev} from Lemma \ref{lemma:polynomial_cumulant_bound} applied to the full set of points $x_1,\ldots,x_k, x_{k+1},\ldots,x_{m+2}$. 
Thus, since $r^{2m+4}_\Delta\leq \beta^2_k$, the last two expressions yield  
\begin{align*}
    \int_{\R \setminus J_k}
f_k(x)^2  \ \Poiss(dx)  
 \leq \frac{(2)^{2m+4} (U_\Delta)^{2m} (U_\Delta + D_\Delta)^4}{r^{2m+6}_\Delta}  \, \mathcal{E}_{\nu}[\mu]
\end{align*}
Since $D_\Delta \leq 2U_\Delta$, the last inequality gives 
\begin{align}\label{eqn:stepvi_secondterm_part_1_v2}
    \int_{\R \setminus J_k}
f_k(x)^2  \ \Poiss(dx)  
 \leq \frac{(3)^{4} (2 U_\Delta)^{2m+4} }{r^{2m+6}_\Delta}  \, \mathcal{E}_{\nu}[\mu]
\end{align}
For the rightmost term in the right-hand side of \eqref{eqn:step6_bounds_mass_middle}, let us note that 
\[
 f_k(x_k) f'_k(x_k) = 0.   
\]
Indeed, for $x\neq x_j$ with $j\neq k$, we have 
\begin{align*}
\beta^{2}_k  \, f_k(x)  f'_k(x) 
& =  x \cdot \prod_{j \neq k} (x-x_j)^2 \left( 1 + x \sum_{i \neq k} \frac{1}{x-x_i}  \right), 
\end{align*}
and evaluating $x=x_k$ in the last parenthesis yield  
\begin{equation*}
   1 + x_k  \sum_{\substack{i=1 \\ i \neq k}}^{m} \frac{1}{x_k-x_i} + x_k \sum_{i= m+1}^{m+2} \frac{1}{x_k-x_i}
=    
1 + x_k  \sum_{\substack{i=1 \\ i \neq k}}^{m} \frac{1}{x_k-x_i} - x_k (s_1 + s_2)
= 0.
\end{equation*}
Thus, setting $g_k(x)=f_k(x)^2 - 1$ for $x \in \R$ and  $\tilde{D_k} =  \frac{1}{2}\sup_{x \in J_k} \abs{ g''_k(x)}$, Taylor's theorem implies 
\[
\abs{f_k(x)^2 - 1} \leq \tilde{D}_k \abs{x-x_k}^2 
\text{\quad for \quad } x \in J_k. 
\]
Hence, we obtain  
\begin{equation*}
    \int_{J_k}\left| 
        f_k(x)^2 - 1 \right|  \ \Poiss(dx)
\leq 
\tilde{D}_k
\int_{J_k}|x-x_k|^{2} \ \Poiss(dx) 
\leq 
\tilde{D}_k \, {\color{black} \radio^{-2m}_\Delta}
\int_{J_k} x^2 (x-x_1)^2 \cdots (x-x_m)^2 \ \Poiss(dx) 
\end{equation*}
where the last inequality follows from the fact that $1 \leq \abs{x} / \radio_{\Delta} $ and $1 \leq \abs{x-x_{\ell}} / \radio_{\Delta} $ for any $x \in J_k$ and ${\ell} \neq k$. 
It follows from the last expression and Lemma \ref{lemma:polynomial_cumulant_bound} that 
\begin{align}\label{eqn:stepvi_secondterm_part_2}
&    \int_{J_k}\left| 
        f_k(x)^2 -1 \right|  \ \Poiss(dx)
 \leq \tilde{D}_k
 \left( \frac{ 2 \, U_\Delta}{  \radio_\Delta } \right)^{2m}  
 \mathcal{E}_{\nu}[\mu]
% \leq \tilde{D}_k \left( \frac{ 2 \, {\color{black} D}_\Delta}{  \radio_\Delta } \right)^{2m+6}  {\color{black} \mathcal{E}_{\nu}[\mu]} 
\end{align}
From \eqref{eqn:step6_bounds_mass_middle}, and putting together \eqref{eqn:step_vi_bound_1_pre}, \eqref{eqn:stepvi_secondterm_part_1_v2} and \eqref{eqn:stepvi_secondterm_part_2}, we get 
\begin{equation}\label{eqn:bound_for_Delta_eta_in_J_k}
|\Delta[J_k]-\Poiss_5[J_k]| \leq 
\left\{ {\textstyle  \tilde{C}_k
\left[ \left( \frac{2 \, U^{}_\Delta}{ r_\Delta } \right)^{2m}   + 1 \right]
+
\frac{(3)^{4} (2 U_\Delta)^{2m+4} }{r^{2m+6}_\Delta} 
+ 
\tilde{D}_k
 \left( \frac{ 2 \, U_\Delta}{  \radio_\Delta } \right)^{2m}   } \right\}
 \mathcal{E}_{\nu}[\mu] 
\end{equation}

Let us now find explicit upper bounds for $\tilde{C}_k$ and $\tilde{D}_k$ that do not depend on $k$. 
First, from \eqref{def:U_Delta_L_Delta} and  \eqref{eqn:bound_for_extra_points}, we have 
\[
%d^{m+2}_\Delta \leq \abs{\beta_k} \text{\quad and \quad} 
|\tilde{\alpha}_{{\color{black}2m+4-n}}| \leq  \sum_{ \tilde{\mathbb{i}}_k } \left( \prod_{\ell \neq k,m+1,m+2} U^2_\Delta \right) (D_\Delta+U_\Delta)^4   
= (2)^{2m+2} (U_\Delta)^{2m-2}  (D_\Delta+U_\Delta)^4   
%\leq (2 D_\Delta )^{2m+6};
%(2)^{2m+2} (D_\Delta)^{2m-2}  (2 D_\Delta)^4  ; 
\]
and
\[
(x_1 \cdots x_{k-1}  x_{k+1} \cdots x_m x_{m+1} x_{m+2})^2 \leq (U^{}_\Delta)^{2m-2}(D_\Delta+U_\Delta)^4 %\leq (2 D_\Delta )^{2m+6} .
\]
additionally,  from \eqref{def:D_Delta_d_Delta} and \eqref{eqn:bound_for_extra_points}, we also get 
\[
r^{m+2}_{\Delta}\leq d^{m+2}_\Delta \leq \abs{\beta_k} = \abs{x_k \cdot {\textstyle \prod_{\substack{j=1 \\ j\neq k}}^{m+2} (x_k-x_j) }}.
\]
Therefore, since $\radio_\Delta \leq d_\Delta \leq 1$, we obtain   
\begin{equation*}
\tilde{C}_k = \beta^{-2}_k \cdot \max\left\{ { \textstyle \prod^{m+2}_{\substack{ j=1 \\ j \neq k }} x^2_j } , \abs{\tilde{\alpha}_3}, \abs{\tilde{\alpha}_4}, \ldots ,  \abs{\tilde{\alpha}_{2m+4}} \right\}
\leq 
\frac{(2)^{2m+2} (U_\Delta)^{2m-2} (U_\Delta + D_\Delta)^4}{r_\Delta^{2m+4}}
.
\end{equation*}
Since $D_\Delta \leq 2U_\Delta$, the last inequality gives 
\begin{equation}\label{eqn:bound_for_constant_Ck}
\tilde{C}_k 
 \leq \frac{(3)^{4} (2 U_\Delta)^{2m+2} }{r^{2m+4}_\Delta} .
\end{equation}
Second, setting $x_0 = 0$, the function $f_k(x)$, $f'_k(x)$, and $f''_k(x)$ can be written as 
\[
f_k(x) = \beta^{-1}_k \prod_{\substack{j=0 \\  j \neq k }}^{m+2} (x-x_j),
\quad 
f'_k(x) = \beta^{-1}_k \sum_{\substack{j=0 \\  j \neq k }}^{m+2}  \prod_{\substack{i=0 \\  i \neq k,j }}^{m+2} (x-x_i),
\quad\text{and}\quad
f''_k(x) = \beta^{-1}_k \sum_{\substack{j=0 \\  j \neq k }}^{m+2}  \sum_{\substack{i=0 \\  i \neq k,j }}^{m+2} \prod_{\substack{\ell=0 \\  i \neq k,j,i }}^{m+2} (x-x_\ell).
\]
But, for any $x \in J_k$, and $j=0,1,\ldots,k-1,k+1,\ldots,m+2$, we know that  
\[
\abs{x-x_j} \leq \abs{x-x_k} + \abs{x_k - x_j} \leq r_\Delta + D_\Delta \leq r_\Delta + 2 U_\Delta.
\]
Thus, for $x \in J_k$ and $\ell = 0,1,2$, we obtain   
\[
| f^{(\ell)}_k (x) | \leq |\beta^{-1}_k|  (m+2)^{\ell} (r_\Delta + 2 U_\Delta)^{m+2-\ell}. 
\]
Since $g''_k = 2 (f''_kf_k + f'_kf'_k)$ , the above upper bounds for $ f^{(\ell)}_k (x)$ with $x \in J_k$ imply 
\begin{equation}\label{eqn:bound_for_constant_Dk}
\tilde{D}_k =  \frac{1}{2}\sup_{x \in J_k} \abs{ g''_k(x)}
\leq  \beta^{-2}_k  2(m+2)^2 (2U_\Delta+r_\Delta)^{2m+2}
\leq \frac{2(m+2)^2 (2 U_\Delta+r_\Delta)^{2m+2}}{r^{2m+4}_\Delta}
\end{equation}
since $r^{2m+4}_\Delta\leq \beta^2_k$.

Therefore, from \eqref{eqn:bound_for_Delta_eta_in_J_k} and the upper bounds \eqref{eqn:bound_for_constant_Ck} and \eqref{eqn:bound_for_constant_Dk}, we obtain  
\[
\abs{\Delta[J_k]-\Poiss_5[J_k]  }  
\leq N^2_\Delta \,
 \left( \frac{ 2 \, U_\Delta}{  \radio_\Delta } \right)^{2m}  
 \,
 \mathcal{E}_{\nu}[\mu] 
\]
where 
\begin{align}\label{eq:Ndeltadef}
N^2_\Delta 
= 
\frac{(3)^{4} ( 2 U_\Delta)^{2} }{r_\Delta^{4}}
\left[ \left( \frac{2 \, U^{}_\Delta}{ r_\Delta } \right)^{2m}   + 1 \right]
+
\frac{(3)^{4} (2 U_\Delta)^{4} }{r^{6}_\Delta} 
+ 
\frac{2(m+2)^2 (2 U_\Delta+r_\Delta)^{2m+2}}{r^{2m+4}_\Delta}  .
\end{align}
Thus, it follows from \eqref{ineq:keyinstep6} and \eqref{eqn:upper_bound_R_vs_Jk} that 
\begin{equation}\label{eqn:wasserstein_triangle_6}
 \freeWass{1} ( \mu_5, \mu_6 )  
 \leq 
 \left( 
\frac{m_2[\Poiss_5 ]}{\Poiss_5[\R]}
\cdot
\abs{\Poiss_5[\R] - \Delta[\R] }  \right)^{1/2}
\leq 
m \, U_\Delta  \, N_\Delta \,
 \left( \frac{ 2 \, U_\Delta}{  \radio_\Delta } \right)^{m}  
 \,
 \mathcal{E}_{\nu}[\mu]^{1/2} 
\end{equation}

%-----------------------------------------------------%
%-------------------- Step VII -----------------------%
%-----------------------------------------------------%

\noindent\textbf{Step VII.} %\\ 
Finally, we compare $\mu_6 = g_{\zeta^{2}} \circledast \pi_{\Poiss_6}  \circledast \delta_{b}  $ and $\mu_7 = g_{\zeta^2}  \circledast \pi_{\Delta} \circledast \delta_b = \nu$ where $\Poiss_6=c\Poiss_5$ with $ c= \Delta[\R]  /  \Poiss_5[\R] $. 
Since $\Poiss_6$ and $\Delta$ have the same total mass, i.e., $\Poiss_6[\R]= \Delta[\R]$, Lemmas \ref{lemma:convolution_wasserstein-inequality} and \ref{lemma:convolution_inequality} imply   
\begin{equation*}
  \freeWass{1} ( \mu_6, \mu_7 ) 
  \leq  \freeWass{1}( \pi_{ \Poiss_6 }  , \pi_{\Delta}  ) 
  \leq  \classWass{1}  ( \Poiss_6  , \Delta) 
  {\color{black}+ 
  |m_1[\Poiss_6]-m_1[\Delta]|}.
\end{equation*}
Moreover, from the dual formulation of $\classWass{1}(\Poiss_6,\Delta)$, we obtain 
\begin{align*}
\classWass{1}(\Poiss_6,\Delta)
%  &=\sup_{f\in Lip_1(\R)}\left|\int_{\R}f(x)(\Poiss_6-\Delta)(dx)\right|\\
  &
  =\sup_{f\in Lip_1(\R)}\left|\int_{\R}(f(x)-f(0))(\Poiss_6-\Delta)(dx)\right| 
\end{align*}
where $Lip_1(\R)$ denotes the set of real $1$-Lipschitz functions defined on $\R$.
Now, both $\Delta$ and $\eta_5$ are supported on $\{x_1,\dots, x_{m}\}$ with $\Delta[\{x_k\}] = \Delta[J_k]$ and $\Poiss_6[\{x_k\}] = \Poiss_6[J_k]$, so the Lipschitz condition implies 
\begin{equation*}
\classWass{1}(\Poiss_6,\Delta)
\leq 
    \left( \max_k \abs{x_k} \right) \sum_{k=1}^{m}|\Delta[J_k]-\Poiss_6[J_k]|
\leq 
    U_\Delta \sum_{k=1}^{m}|\Delta[J_k]-\Poiss_6[J_k]|.
\end{equation*}
Moreover, we have 
\begin{align*}
|m_1[\Poiss_6]-m_1[\Delta]|
&=
\left|
\sum_{k=1}^{m}x_k\big(\Poiss_6[J_k]-\Delta[J_k]\big)
\right| \leq
U_\Delta \sum_{k=1}^{m}|\Delta[J_k]-\Poiss_6[J_k]|. 
\end{align*}
Consequently, we obtain 
\begin{equation}\label{eqn:step_vii_first_ineq}
\freeWass{1}(\mu_6,\mu_7)
\leq
2U_\Delta
\sum_{k=1}^{m}|\Delta[J_k]-\Poiss_6[J_k]|.
\end{equation}
Now, let us find an upper bound for each $|\Delta[J_k]-\Poiss_6[J_k]|$. 
For this, consider the inequality 
\begin{equation*}\label{eqn:bound_for_Delta_eta_in_J_k_step_vi} 
|\Delta[J_k]-\Poiss_6[J_k]| 
\leq
\abs{\Delta[J_k]-\Poiss_5[J_k]  } + 
\abs{\Poiss_5[J_k]-\Poiss_6[J_k] } 
\end{equation*}
and recall that 
\[
\abs{\Delta[J_k]-\Poiss_5[J_k]  }  
\leq N^2_\Delta \,
 \left( \frac{ 2 \, U_\Delta}{  \radio_\Delta } \right)^{2m}  
 \mathcal{E}_{\nu}[\mu] 
 \text{\quad and \quad} 
 |\Delta[\R]-\Poiss_5[\R]|  \leq  
m \, N^2_\Delta \,
 \left( \frac{ 2 \, U_\Delta}{  \radio_\Delta } \right)^{2m}  
 \mathcal{E}_{\nu}[\mu] 
\]
from the previous step. 
Moreover, since $\Poiss_6=c\Poiss_5$ with $ c= \Delta[\R]  /  \Poiss_5[\R] $ and $\eta_5 [J_k] \leq \eta_5[\R]$,  we have  
\begin{align*}
\abs{\Poiss_5[J_k]-\Poiss_6[J_k] }
=
\frac{\Poiss_5[J_k]}{\Poiss_5[\R]} \cdot 
\abs{ \Poiss_5[\R] - \Delta[\R]   }
\leq 
\abs{ \Poiss_5[\R] - \Delta[\R]   }
\end{align*} 
Thus, we get 
\[
\abs{\Delta[J_k]-\Poiss_5[J_k]  }   
\leq (m+1) \, N^2_\Delta \,
 \left( \frac{ 2 \, U_\Delta}{  \radio_\Delta } \right)^{2m}  
 \,
 \mathcal{E}_{\nu}[\mu], 
\]
and therefore 
\begin{align}\label{eqn:wasserstein_triangle_7}
\freeWass{1}(\mu_6,\mu_7)
& \leq 
2U_\Delta  \, m \, (m+1) \, N^2_\Delta \, 
 \left( \frac{ 2 \, U_\Delta}{  \radio_\Delta } \right)^{2m}  
  \mathcal{E}_{\nu}[\mu] \nonumber \\ 
& \leq 
2U_\Delta  \, m \, (m+1) \, N^2_\Delta \, 
 \left( \frac{ 2 \, U_\Delta}{  \radio_\Delta } \right)^{2m}  
  \mathcal{E}_{\nu}[\mu]^{1/2}
\end{align}
where the last inequality follows from the assumption that $\mathcal{E}_{\nu}[\mu] \leq 1$.  

Finally, putting together  \eqref{eqn:wasserstein_triangle_1}, \eqref{eqn:wasserstein_triangle_2}, 
\eqref{eqn:wasserstein_triangle_3},
\eqref{eqn:wasserstein_triangle_4}, 
\eqref{eqn:wasserstein_triangle_5},
\eqref{eqn:wasserstein_triangle_6}, \eqref{eqn:wasserstein_triangle_7}, 
we conclude 
\begin{equation}
\freeWass{1} (\mu,\nu) \leq \left\{ 1 + (M_\Delta +
2 m_2[\mu]^{1/2} ) \left( \frac{ 2 \, U_\Delta}{  \radio_\Delta } \right)^{m}  \right\} 
  \mathcal{E}_{\nu}[\mu]^{1/2}
\end{equation}
where 
\begin{equation}
M_\Delta
= 
3 
+
m \, U_\Delta  \, N_\Delta 
+
2U_\Delta  \, m \, (m+1) \, N^2_\Delta \, 
 \left( \frac{ 2 \, U_\Delta}{  \radio_\Delta } \right)^{m}  
\end{equation}

%removecomment\newpage 

%\nocite{*}
\bibliographystyle{abbrv}
\bibliography{bibibi}

\end{document}